\documentclass[11pt, oneside]{report}

\usepackage{fullpage}
\usepackage[hang,flushmargin,symbol*]{footmisc}
\usepackage{graphicx}
\usepackage{enumitem}
\usepackage{subfigure}
\usepackage{amsmath}
\usepackage{amsthm}
\usepackage{amssymb}
\usepackage{mathtools}
\usepackage{color}
\definecolor{darkblue}{rgb}{0, 0, .6}
\definecolor{grey}{rgb}{.7, .7, .7}
\usepackage[breaklinks]{hyperref}
\hypersetup{
	colorlinks=true,
	linkcolor=darkblue,
	anchorcolor=darkblue,
	citecolor=darkblue,
	pagecolor=darkblue,
	urlcolor=darkblue,
	pdftitle={},
	pdfauthor={}
}

\newtheorem{theorem}{Theorem}[chapter] 
\newtheorem{lemma}[theorem]{Lemma}
\newtheorem{proposition}[theorem]{Proposition}
\newtheorem{corollary}[theorem]{Corollary}

\theoremstyle{definition} 
\newtheorem{definition}[theorem]{Definition}
\newtheorem{example}[theorem]{Example}

\newcommand{\A}{\ensuremath{\mathcal A}}

\renewcommand{\to}{\longrightarrow}
\newcommand{\C}{\ensuremath{\mathbb C}}
\newcommand{\R}{\ensuremath{\mathbb R}}

\newcommand{\Sal}{\operatorname{Sal}}
\newcommand{\Real}{\operatorname{Re}}
\newcommand{\Imag}{\operatorname{Im}}
\newcommand{\Bd}{\operatorname{Bd}}
\newcommand{\Int}{\operatorname{Int}}
\newcommand{\Ord}{\operatorname{Ord}}

\newcommand{\Star}{\operatorname{Star}}
\renewcommand{\bar}[1]{\ensuremath{\overline{#1}}}

\begin{document}

\setlength{\parindent}{0pt}


\begin{titlepage}
\ 
\vspace{5cm}

{\huge \textbf{Cell Complexes for Arrangements with\\
Group Actions\footnote{This is a revised version of the author's M.S. thesis, which was directed by Michael J. Falk at Northern Arizona University.  All revisions are cosmetic.}}}

\bigskip

M.S. Thesis, Northern Arizona University, 2000

\bigskip

\bigskip

\href{http://oz.plymouth.edu/~dcernst}{\Large Dana C. Ernst}\\
Plymouth State University\\
Department of Mathematics\\
MSC 29, 17 High Street\\
Plymouth, NH 03264\\
\href{mailto:dcernst@plymouth.edu}{{dcernst@plymouth.edu}}

\end{titlepage}

\pagenumbering{roman}
\pagestyle{plain}


\chapter*{Abstract}\normalsize
\addcontentsline{toc}{chapter}{Abstract}

For a real oriented hyperplane arrangement, we show that the corresponding Salvetti complex is homotopy equivalent to the complement of the complexified arrangement.  This result was originally proved by M. Salvetti.  Our proof follows the framework of a proof given by L. Paris and relies heavily on the notation of oriented matroids.  We also show that homotopy equivalence is preserved when we quotient by the action of the corresponding reflection group.  In particular, the Salvetti complex of the braid arrangement in $\ell$ dimensions modulo the action of the symmetric group is a cell complex which is homotopy equivalent to the space of unlabelled configurations of $\ell$ distinct points.  Lastly, we describe a construction of the orbit complex from the dual complex for all finite reflection arrangements in dimension 2.  This description yields an easy derivation of the so-called ``braid relations" in the case of braid arrangement.

\chapter*{Acknowledgements}\normalsize
\addcontentsline{toc}{chapter}{Acknowledgements}

I would like to thank my advisor Dr. Michael Falk for providing me with a topic 
that completely captured my attention.  I would also like to thank him for 
having an infinite amount of patience with me and allowing me to spend so much 
time in the chair in the corner of his office.  Without him this thesis would 
not have been possible.  I would also like to thank Dr. Janet McShane and Dr. 
Guenther Huck for taking the time to read my thesis and providing me with useful 
feedback.  

\tableofcontents

\listoffigures


\pagenumbering{arabic}

\begin{chapter}{Posets and Cell Complexes}

\begin{paragraph}{Cell Complexes}
A topological space $E$ is called a \emph{cell} of dimension $n$ if it is 
homeomorphic to the unit disc, $B^n$.  In this case we say that $E$ is an 
$n$-cell.  A space $F$ of $\R^\ell$ is called an \emph{open cell} of dimension 
$n$ if it is homeomorphic to $\Int(B^n)$.  In this case we say that $F$ is an 
open $n$-cell.  We will refer to a $0$-cell as a vertex.  

\begin{definition}
We define a (finite) \emph{cell complex} to be a Hausdorff topological 
space $X$ together with a finite collection $F_0, \ldots, F_n$ of disjoint open 
cells in $X$ such that
\begin{enumerate}[label=\rm{(\arabic*)}]
\item  $X = \bigcup_{k=0}^n F_k$.
\item For each open $n$-cell $F_k$, there exists a continuous map 	
	$$f_k: B^n \to X$$
that maps $\Int(B^n)$ homeomorphically onto $F_k$ and carries $\Bd(B^n)$ into a 
union of open cells, each of dimension less than $n$.
\end{enumerate}
\noindent Note that what we are calling a cell complex is often referred to as a 
finite CW complex.
\end{definition}

\begin{example}\label{ex1.1}
The quotient space formed from $B^\ell$ by collapsing $\Bd(B^\ell)$ to a point 
is homeomorphic to $S^\ell$.  Therefore, $S^\ell$ can be expressed as a cell 
complex having one open $\ell$-cell and one $0$-cell, and no other cells at all.
\end{example}

A cell complex $X$ for which the maps $f_k$ can be taken to be homeomorphisms is 
called a \emph{regular cell complex}.  

\begin{example}
The cell complex described in Example \ref{ex1.1} is not a regular cell complex. 
However, we can give a regular cell complex structure to $S^\ell$.  Consider 
$S^1$.  We can subdivide $S^1$ into two open $1$-cells and two $0$-cells.  See 
Figure \ref{s1}.  The resulting cell complex is regular.  Likewise, consider 
$S^2$.  Take $S^1$ to be the equator of $S^2$ and let $S^1$ have the regular 
cell structure described above.  Then $S^2$ will be a regular cell complex 
having two open $2$-cells, two open $1$-cells, and two $0$-cells.  Similarly, we 
can give $S^\ell$ a regular cell complex structure by subdividing $S^\ell$ into 
two open $\ell$-cells, i.e., the open upper and lower hemispheres, and giving 
the equator the regular cell structure of $S^{\ell-1}$.
\begin{figure}[h]
\begin{center}

\includegraphics[width=4cm]{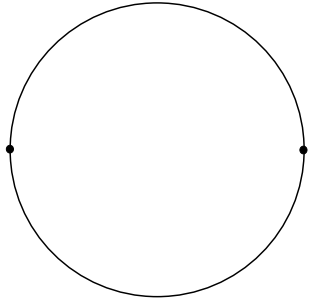}

\caption{$S^1$ as a regular cell complex}
\label{s1}
\end{center}
\end{figure}
\end{example}

We define the \emph{$k$-skeleton} of $X$ to be the subspace $X^k$ of $X$ that is 
the union of the open cells of $X$ of dimension at most $k$.  Note that if the 
highest dimensional cells of $X$ are of dimension $n$, then $X=X^n$. Also, we 
have
	$$X^0 \subset X^1 \subset \cdots \subset X^n=X.$$
Note that $X^0$ is the vertex set of $X$.  Furthermore, note that $X^k-X^{k-1}$ 
is a disjoint union of open $k$-cells, possibly empty, for each $k$.

\bigskip

We sometimes refer to open cells as \emph{faces}.  Each open cell $F$ of $X$ is 
a face of $X$.  We denote the set of faces of $X$ by $\mathcal{F}(X)$, which is 
partially ordered via
	$$F' \leq F \ \text{iff} \ F' \subseteq \bar{F},$$
for all $F',F \in \mathcal{F}(X)$.  We will refer to $\mathcal{F}(X)$ as the 
face poset of $X$.

\begin{example}
Let $X$ be the regular cell complex given in Figure \ref{cell}.  Then the 
cells of $X$ determine the face poset, $\mathcal{F}(X)$, given in Figure 
\ref{poset}.  
\begin{figure}[h]
\begin{center}

\includegraphics[width=6cm]{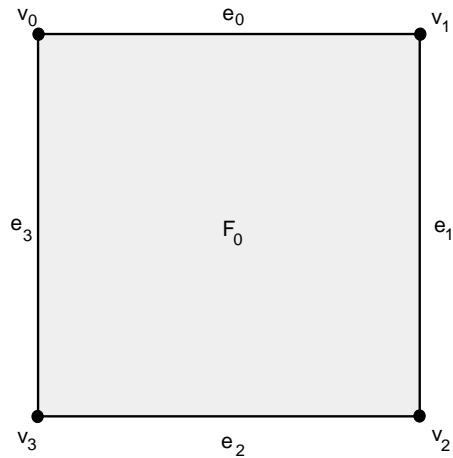}

\caption{A regular cell complex}
\label{cell}
\end{center}
\end{figure}

\begin{figure}[h]
\begin{center}

\includegraphics[width=7cm]{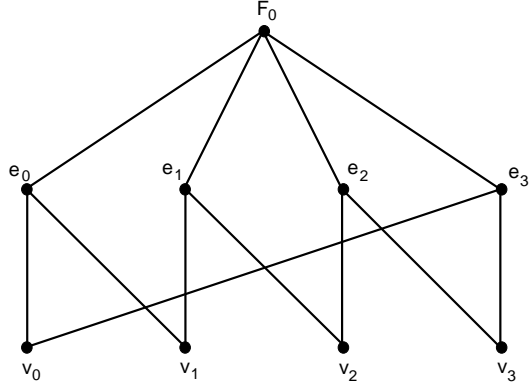}

\caption{Face poset for a regular cell complex}
\label{poset}
\end{center}
\end{figure}
\end{example}

In \cite{BJ}, Bj{\"o}rner characterizes face posets of regular cell 
complexes.  One consequence of his result is stated below without proof.

\begin{proposition}
Let $X$ be a regular cell complex.  Then $X$ is uniquely determined up to 
cellular homeomorphism by its face poset $\mathcal{F}(X)$. \hfill $\Box$
\end{proposition}
\end{paragraph}

\begin{paragraph}{Concrete Simplicial Complexes}

Now, we will discuss a special class of regular cell complexes, called 
simplicial complexes.  First, recall that a set $\{v_0, \ldots, v_n\}$ of points 
in a vector space $V$ is said to be \emph{geometrically independent} iff the 
vectors
	$$v_1-v_0, \ldots, v_n-v_0$$
are linearly independent.  So, two distinct points in $\R^\ell$ form 
a geometrically independent set, as do three non-collinear points, four 
non-coplanar points, and so on.  

\bigskip

If $\{v_0, \ldots, v_n\}$ is a geometrically independent set of points in 
$\R^\ell$, then we define the \emph{concrete $n$-simplex} $\sigma$ spanned by 
$v_0, \ldots, v_n$ to be the convex hull of $v_0, \ldots, v_n$, denoted 
	$$v_0 \vee \cdots \vee v_n.$$
The points $v_0, \ldots, v_n$ are called the vertices of $\sigma$.  

\bigskip

In low dimensions, one can picture a concrete simplex easily. A $0$-simplex is 
just a point, i.e., a vertex.  A $1$-simplex spanned by $v_0$ and $v_1$ is just 
the line segment joining $v_0$ and $v_1$.  A $2$-simplex spanned by $v_0$, 
$v_1$, and $v_2$ is the solid triangle having these three points as vertices.  
Similarly, a $3$-simplex is a tetrahedron.  For $n \geq 4$, an $n$-simplex is 
the $n$-dimensional analogue of a tetrahedron.

\bigskip

Any simplex spanned by a subset of $\{v_0, \ldots, v_n\}$, where $v_0, \ldots, 
v_n$ are the vertices of a simplex $\sigma$, is called a \emph{face} of 
$\sigma$.  The union of the faces of $\sigma$ different from $\sigma$ itself is 
called the \emph{boundary} of $\sigma$ and denoted $\Bd(\sigma)$.  The 
\emph{interior} of $\sigma$ is defined by the equation 
	$$\Int(\sigma) = \sigma - \Bd(\sigma).$$  
The interior of a simplex $\sigma$ is called an \emph{open simplex}.  Note that 
if the vertices $v_0, \ldots, v_n$ span a simplex $\phi$, then we will sometimes 
use
	$$v_0 \vee \cdots \vee v_n$$
to denote the open simplex spanned by $v_0, \ldots, v_n$, making sure to clarify 
the meaning.

\begin{definition}
A \emph{concrete simplicial complex} $\Delta$ in $\R^\ell$ is a collection of 
simplices in $\R^\ell$ such that
\begin{enumerate}[label=\rm{(\roman*)}]
\item  Every face of a simplex of $\Delta$ is in $\Delta$.
\item  The intersection of any two simplexes of $\Delta$ is a face of each 
of them.
\end{enumerate}
\end{definition}

\begin{example}
The collection $\Delta_1$ pictured in Figure \ref{delta}(a), consisting of two 
$2$-simplices with an edge in common, along their faces, is a simplicial 
complex.  However, the collection $\Delta_2$ pictured in Figure \ref{delta}(b) 
is not a simplicial complex.
\begin{figure}[h]
\begin{center}

\mbox{\subfigure[$\Delta_1$]{\includegraphics[width=5cm]{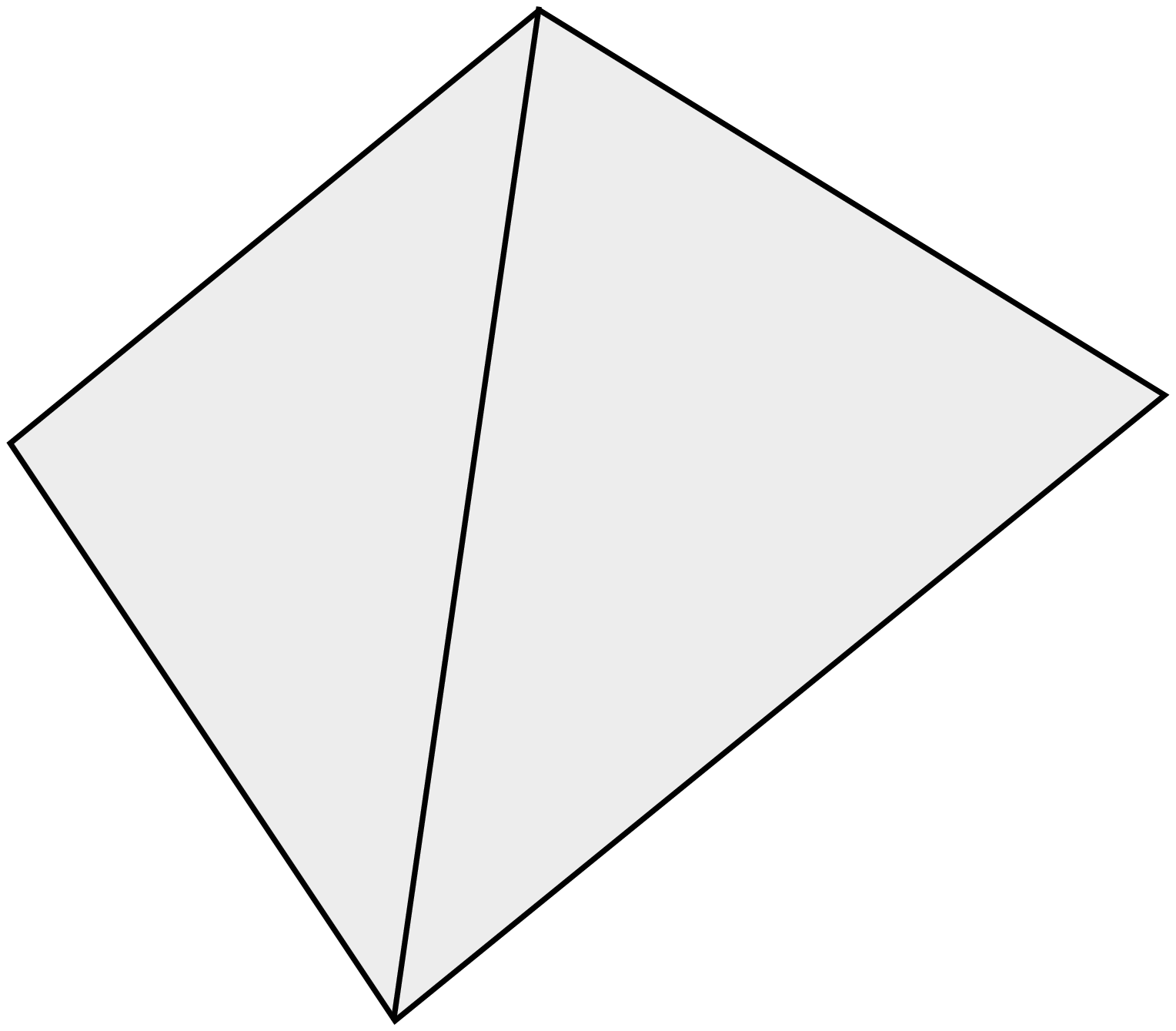}}\quad
\subfigure[$\Delta_2$]{\includegraphics[height=5cm]{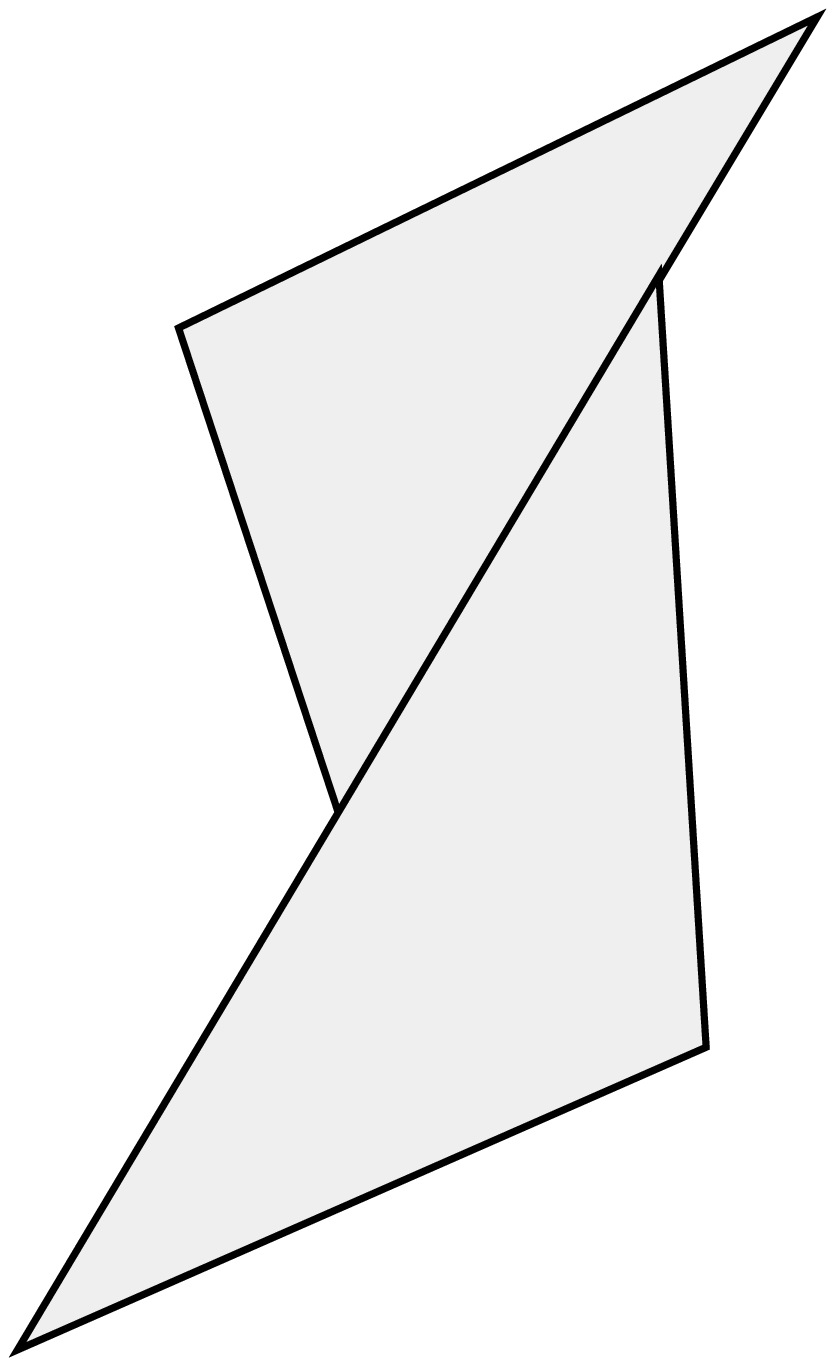}}}

\caption{Collections of simplices}
\label{delta}
\end{center}
\end{figure}
\end{example}

It is clear that a simplicial complex $\Delta$ is a regular cell complex, where 
the cells are the simplices.  The collection of $0$-simplices of $\Delta$ are 
the vertices of $\Delta$, i.e., $\Delta^0$.  We will abuse notation and use 
$\Delta$ to refer to both the simplicial complex $\Delta$ and the subset of 
$\R^\ell$ that is the union of the simplices of $\Delta$.  Note that the same 
subset of $\R^\ell$ can have different simplicial structures.

\begin{definition}
If $v$ is a vertex of a simplicial complex $\Delta$, then the \emph{star} of $v$ 
in $\Delta$, denoted $\Star(v)$, is the union of the interiors of those 
simplices of $\Delta$ that have $v$ as a vertex.
\end{definition}

\begin{example}
Let $\Delta$ be the simpicial complex given in Figure \ref{star}(a).  Then the 
star of the vertex $v_0$ is shown in Figure \ref{star}(b).
\begin{figure}[h]
\begin{center}

\mbox{\subfigure[$\Delta$]{\includegraphics[width=5cm]{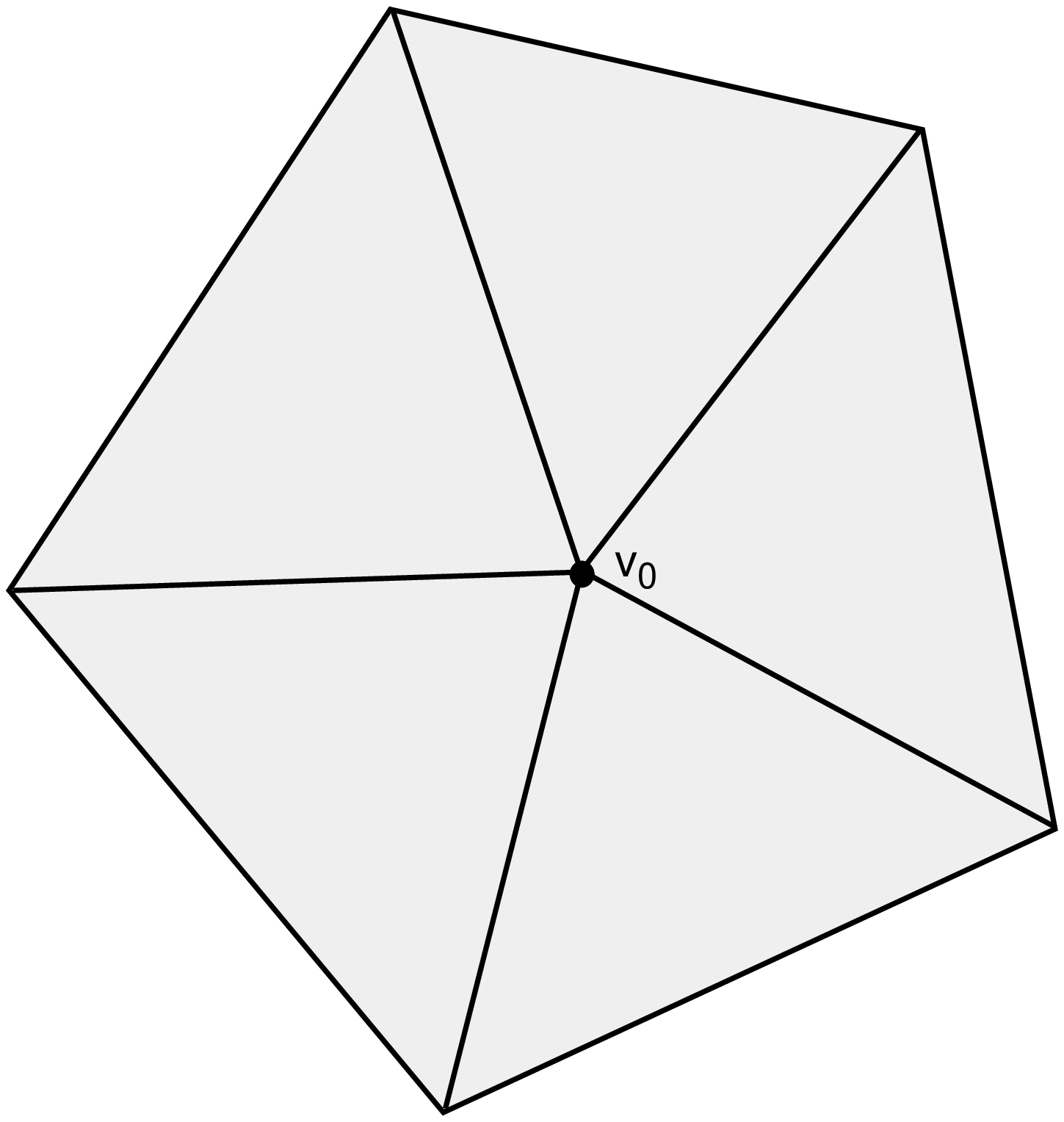}}\quad
\subfigure[$\Star(v_0)$]{\includegraphics[width=5cm]{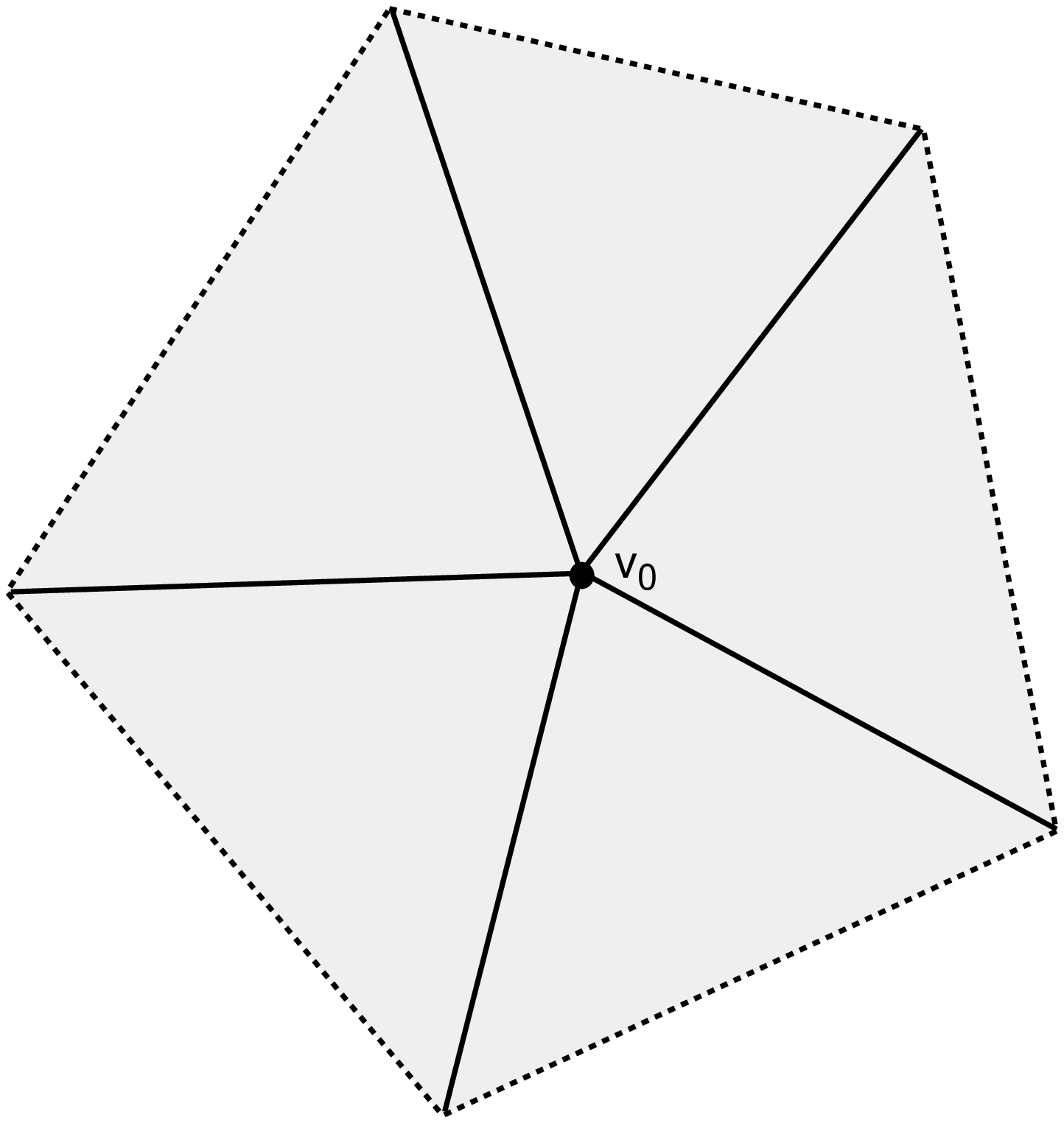}}}

\caption{The star of a vertex}
\label{star}
\end{center}
\end{figure} 
\end{example}

Since a simplicial complex $\Delta$ is a regular cell complex, the simplices of 
$\Delta$ subdivide $\Delta$ into faces.  Each simplex $\sigma$ of $\Delta$ is a 
face of $\Delta$.  Again, we denote the set of faces of $\Delta$ by 
$\mathcal{F}(\Delta)$.   The partial ordering on the face poset 
$\mathcal{F}(\Delta)$ is the same as before.  Furthermore, a simplex $\sigma$ in 
$\mathcal{F}(\Delta)$ is uniquely determined by its set of vertices.

\end{paragraph}

\begin{paragraph}{Abstract Simplicial Complexes}

The partial ordering on simplices of a simplicial complex together with the last 
remark above leads us to the 
following, more general, definition of simplicial complex.

\begin{definition}
An \emph{abstract simplicial complex} is a collection $\mathcal{S}$ of finite 
nonempty sets, such that if $A$ is an element of $\mathcal{S}$, then so is every 
nonempty subset of $A$.
\end{definition}

The element $A$ of $\mathcal{S}$ is called a simplex of $\mathcal{S}$.  If $A$ 
consists of $n+1$ elements, then $A$ is an $n$-simplex, and is said to have 
dimension $n$.  Each nonempty subset of $A$ is called a face of $A$.  By 
hypothesis, each face of $A$ is also a simplex.  The vertex set of $\mathcal{S}$ 
is the union of all the $0$-simplices of $\mathcal{S}$, where the $0$-simplices 
are the singletons of $\mathcal{S}$.  

\bigskip

Every abstract simplicial complex is partially ordered by set inclusion.  We 
say that that two abstract simplicial complexes $\mathcal{S}$ and 
$\mathcal{T}$ are isomorphic if they are isomorphic as posets.

\begin{definition}
Let $\Delta$ be a concrete simplicial complex.  Define $V(\Delta)$ to be the 
collection of all subsets $\{v_0, \ldots, v_n\}$ of $\Delta^0$ such that 
the vertices $v_0, \ldots, v_n$ span an $n$-simplex of $\Delta$.  The 
collection $V(\Delta)$ is called the \emph{vertex scheme} of $\Delta$, and is 
partially ordered by inclusion.
\end{definition}

It should be clear that the vertex scheme for a concrete simplical complex 
$\Delta$ satisfies the criteria for being an abstract simplicial complex.  That 
is, $V(\Delta)$ is an abstract simplicial complex.

\begin{example}
Let $\Delta$ be the concrete simplicial complex of Figure \ref{vertscheme}(a).  
Then the vertex scheme $V(\Delta)$ is given in Figure \ref{vertscheme}(b).
\begin{figure}[h]
\begin{center}

\mbox{\subfigure[$\Delta$]{\includegraphics[width=5cm]{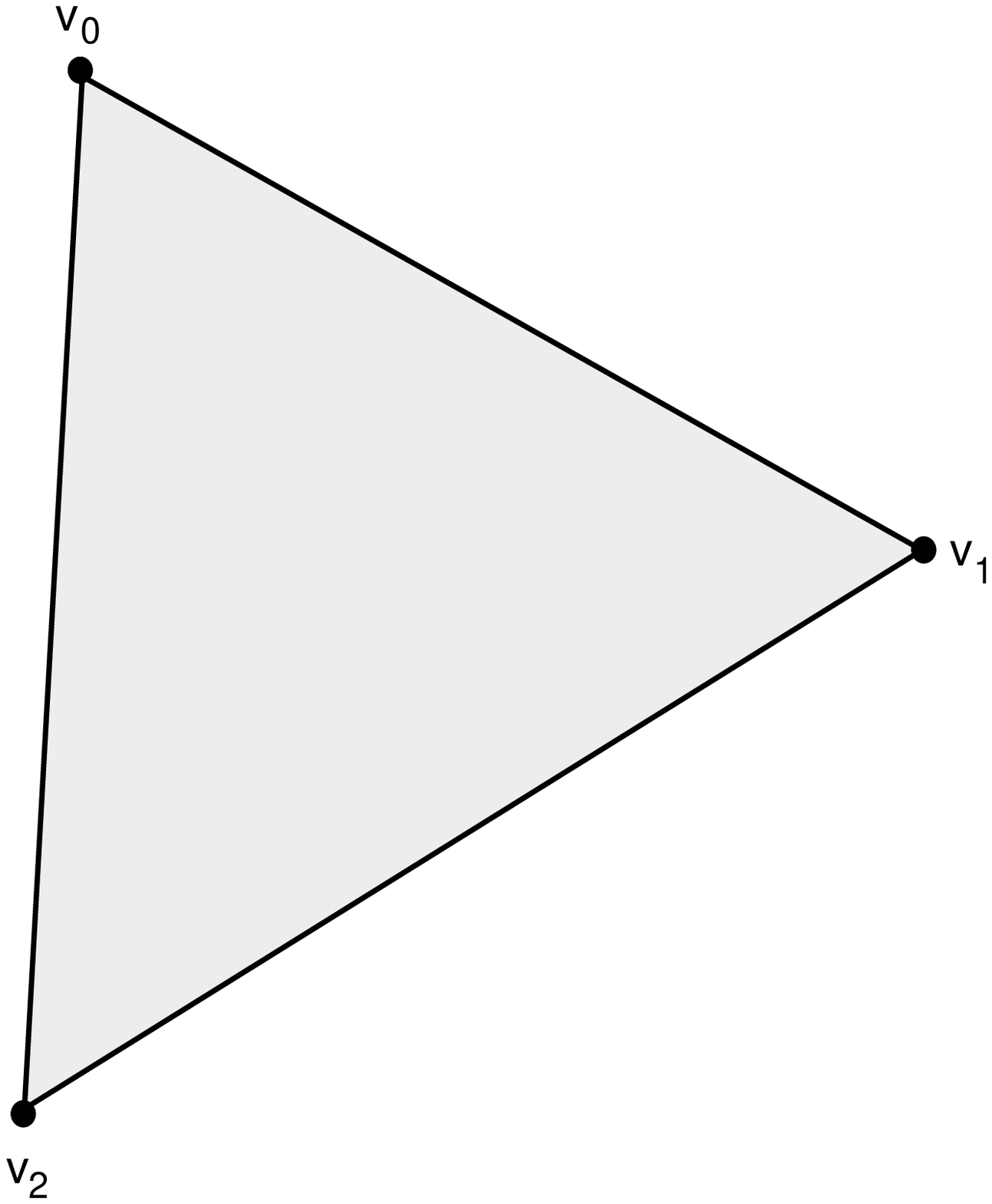}}\quad
\subfigure[$V(\Delta)$]{\includegraphics[width=7cm]{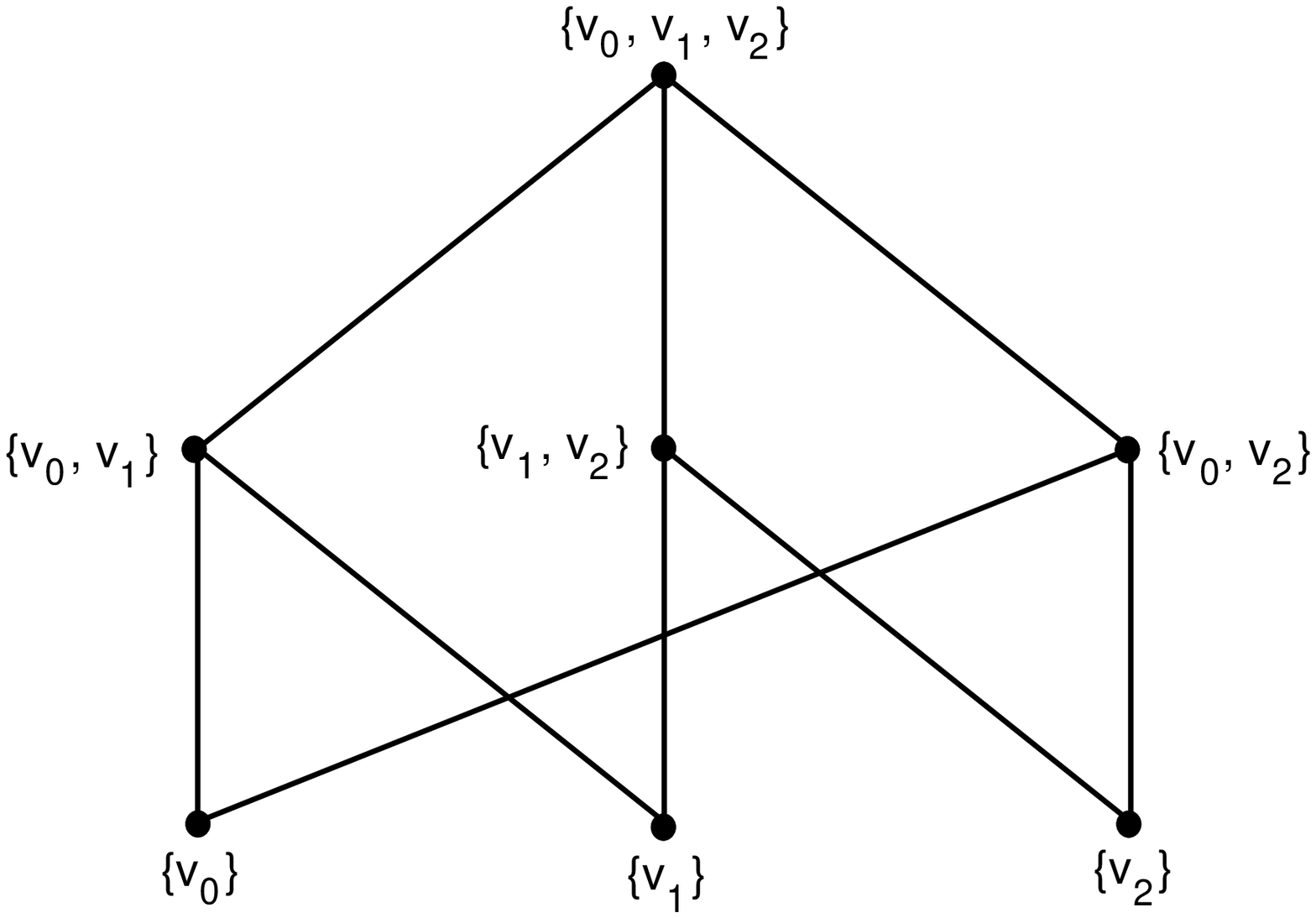}}}

\caption{A simplicial complex and its vertex scheme}
\label{vertscheme}
\end{center}
\end{figure}
\end{example}

Since $\{v_0, \ldots, v_n\}$ is an element of $V(\Delta)$ only when $v_0, 
\ldots, v_n$ span a simplex of $\Delta$, it should be clear that 
$\mathcal{F}(\Delta)$ and $V(\Delta)$ are isomorphic as posets. For this reason, 
we will consider $\mathcal{F}(\Delta)$ to be an abstract simplicial complex. 

\bigskip

Not only is $\mathcal{F}(\Delta)$ an example of an abstract simplicial complex,
it is the crucial example, as the following proposition shows.

\begin{proposition}
Every abstract simplicial complex $\mathcal{S}$ is isomorphic to the vertex 
scheme of some concrete simplicial complex $\Delta$.  That is, $\mathcal{S}$ is 
isomorphic to the face poset of some concrete simplicial complex $\Delta$.  
Furthermore, $\Delta$ is uniquely determined up to simplicial homeomorphism.
\end{proposition}

\begin{proof}
Say that $v_0, \ldots, v_n$ are the vertices of $\mathcal{S}$.  We can easily 
form a realization of $\mathcal{S}$ in $\R^{n+1}$ in the following way.  Let 
$e_k$ denote the $k^{th}$ standard basis vector in $\R^{n+1}$.  Form a concrete 
simplex spanned by a subset of $\{e_0, \ldots, e_n\}$ exactly when the 
corresponding subset of $\{v_0, \ldots, v_n\}$ is an abstract simplex.  The 
resulting concrete simplicial complex $\Delta$ is a realization of 
$\mathcal{S}$.  That is, $V(\Delta)$ is isomorphic to $\mathcal{S}$.
\end{proof}

Note that $\Delta$ in the proof above is a subcomplex of the standard 
$n$-simplex in $\R^{n+1}$. 

\begin{definition}
If the abstract simplicial complex $\mathcal{S}$ is isomorphic to the face poset 
of the concrete simplicial complex $\Delta$, then we call $\Delta$ a realization 
of $\mathcal{S}$.  $\Delta$ is unique up to simplicial homeomorphism.  We will 
denote a realization of $\mathcal{S}$ by $|\mathcal{S}|$.
\end{definition}

\end{paragraph}

\begin{paragraph}{The Order Complex}

Given any poset $P$, we can form its \emph{order complex}, denoted $\Ord(P)$.  
This is an abstract simplicial complex whose vertices are the elements of $P$ 
and whose simplices are the chains 
	$$p_0  < \cdots < p_n$$
in $P$.  We will use the notation
	$$\Omega = (\omega_0 < \cdots < \omega_n)$$
to denote a simplex of the order complex.  Note that $\Ord(P)$ is an abstract 
simplicial complex since any subchain of a chain is also a chain.  Also, note 
that the realization $|\Ord(P)|$ is a concrete simplicial complex.

\bigskip

In the special case where $\mathcal{F}(X)$ is the face poset for a regular cell 
complex $X$, the order complex of $\mathcal{F}(X)$ is called the 
\emph{barycentric subdivision} of $X$, denoted $X'$.  We can form a realization 
of $X'$ in the following way.  

\bigskip

Let $F$ be a cell of a regular cell complex $X$.  Then $\bar{F}$ is homeomorphic 
to $B^\ell$ for some $\ell$.  Also, we can give $B^\ell$ a regular cell 
structure which is isomorphic to the regular cell structure of $\bar{F}$.  Let 
$B_{F'}$ denote the face of $B^\ell$ which corresponds to the face $F'$ of 
$\bar{F}$.  For every open face $F'$ of $\bar{F}$, select a point $x(F') \in 
F'$, called the \emph{barycenter} of $F'$.  Then each 
point $x(F')$ determines a point $x(B_{F'}) \in B_{F'}$, since there is a 
homeomorphism between $F'$ and $B_{F'}$.  Whenever 
	$$F_0 < \cdots < F_n$$
is a chain of faces of $\bar{F}$ form the simplex 
	$$\phi_B = x(B_{F_0}) \vee \cdots \vee x(B_{F_n})$$
on $B^\ell$.  Recall that $\phi_B$ is the convex hull of the vertices 
$x(B_{F_0}), \ldots, x(B_{F_n})$ on $B^\ell$.  Since there is homeomorphism 
between $B^\ell$ and $\bar{F}$, each simplex $\phi_B$ of $B^\ell$ determines the 
simplex $\phi$ of $\bar{F}$ having $x(F_0), \ldots, x(F_n)$ as its vertices.  
Even though $\phi$ may not be the convex hull of the vertices $x(F_0), \ldots, 
x(F_n)$,we will still let
	$$x(F_0) \vee \cdots \vee x(F_n)$$
denote the simplex $\phi$.  The resulting cell complex is homeomorphic to a 
concrete simplicial complex in $\R^\ell$.  

\bigskip

Note that we can form the barycentric subdivision of any regular cell complex 
over and over again.  Each subsequent subdivision creates more simplicial 
cells.

\begin{example}
Consider the regular cell complex $X$ indicated in Figure \ref{cell'}(a).  Its 
first barycentric subdivision, $X'$, is pictured in Figure \ref{cell'}(b).
\begin{figure}[h]
\begin{center}

\mbox{\subfigure[$X$]{\includegraphics[width=5cm]{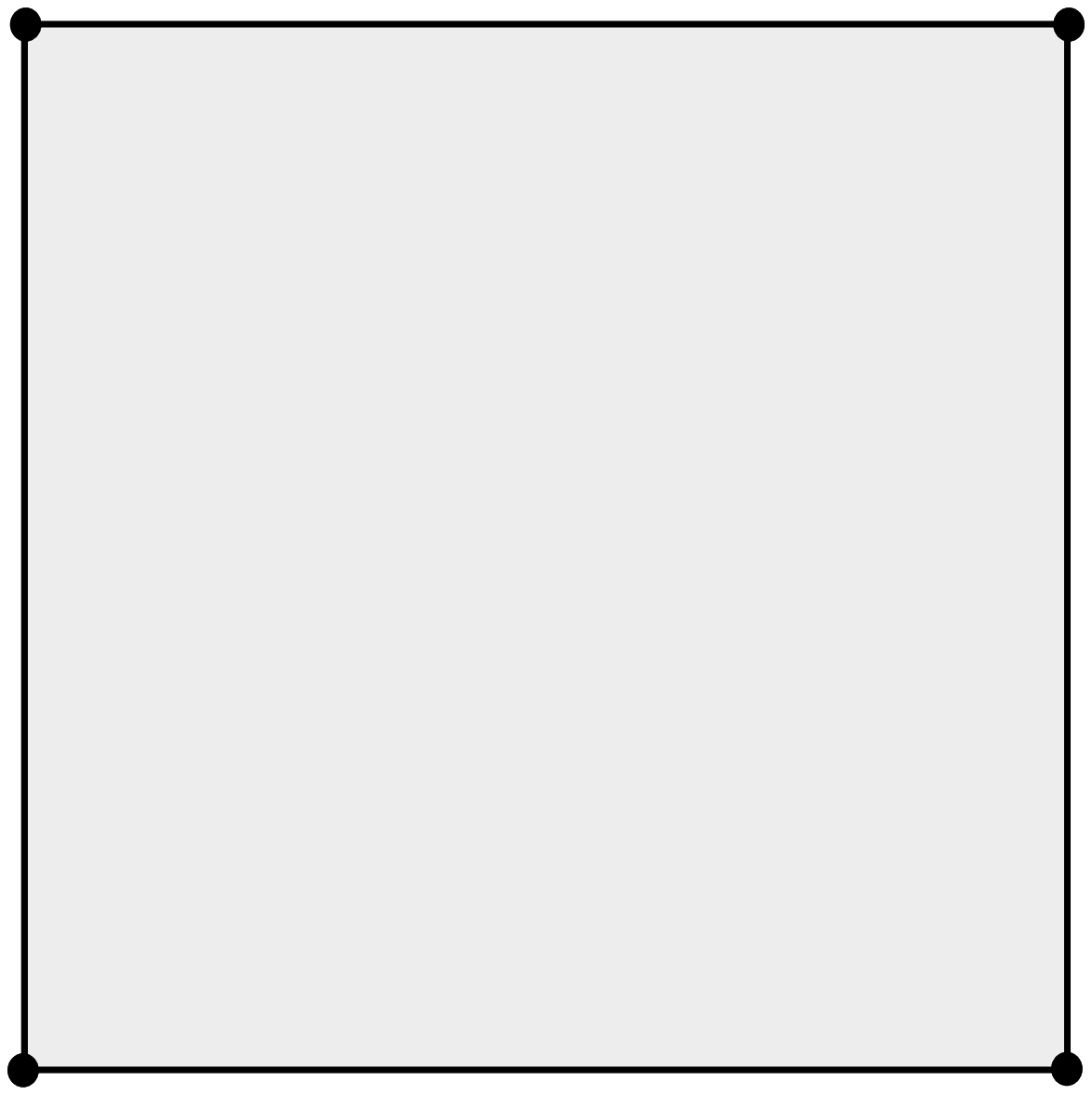}}\quad
\subfigure[$X'$]{\includegraphics[width=5cm]{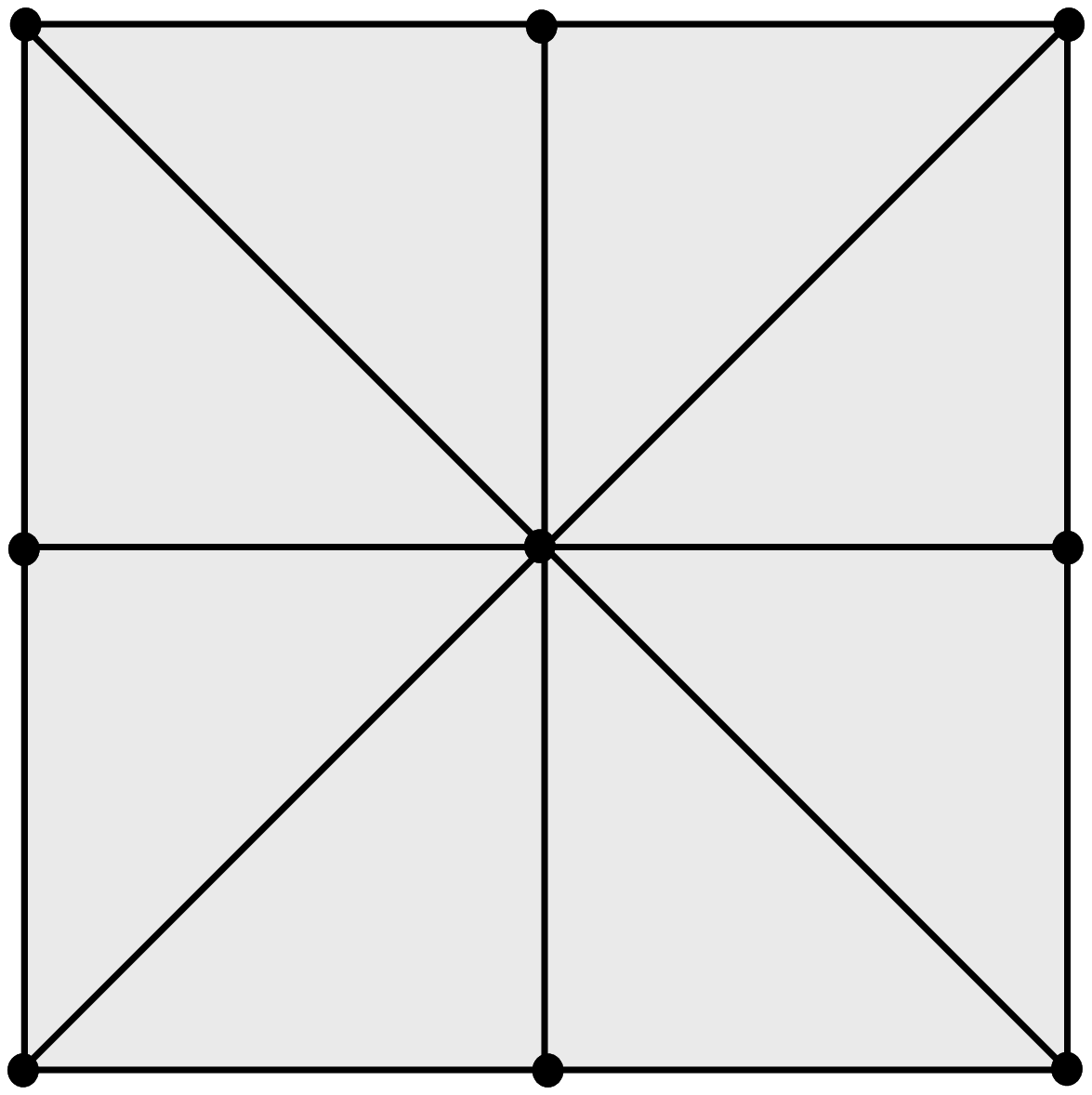}}}

\caption{Barycentric subdivisions of a regular cell complex}
\label{cell'}
\end{center}
\end{figure}
\end{example}

Barycentric subdivision of a regular cell complex $X$ leaves all of its 
topological properties intact.  Hence we have the following proposition.

\begin{proposition}\label{homeo}
If $X$ is a regular cell complex, then $X$ and $X'$ are homeomorphic as 
topological spaces. \hfill $\Box$
\end{proposition}  

We will now define a special example of an abstract simplicial complex which we 
will make use of in later chapters.

\begin{definition}
Let $\mathcal{U}$ be a collection of subsets of $\R^\ell$.  We define an 
abstract simplicial complex called the \emph{nerve} of $\mathcal{U}$, denoted 
$N(\mathcal{U})$.  Its vertices are the elements of $\mathcal{U}$ and its 
simplices are the finite subcollections $\{U_0, \ldots, U_n\}$ such that
	$$U_0 \cap \cdots \cap U_n \neq \emptyset.$$
\end{definition}

Note that $N(\mathcal{U})$ is an abstract simplicial complex, since if
	$$U_0 \cap \cdots \cap U_n \neq \emptyset\text{,}$$
then any subcollection of $\{U_0, \ldots, U_n\}$ also has nonempty intersection.

\bigskip

The following theorem about the nerve of open covers is called the Nerve 
Theorem.  We will not make use of this theorem until Chapter \ref{chap5}.  For a proof, 
see \cite{BO}.

\begin{theorem}\label{nerve}
Let $\mathcal{U}$ be an open cover of contractible sets of a subset $A$ of 
$\R^\ell$ (or $\C^\ell$).  Then $|N(\mathcal{U})|$ has the same homotopy type as 
$A$. \hfill $\Box$
\end{theorem}

For a closer look at this development of regular cell complexes and 
simplicial complexes see \cite{MU}.

\end{paragraph}

\end{chapter}


\begin{chapter}{Hyperplane Arrangements}\label{chap2}

By a \emph{hyperplane} $H$ in $\R^\ell$ we mean a linear subspace $H \subseteq 
\R^\ell$ of codimension 1.  That is, we are assuming a hyperplane $H$ contains 
the origin $\{0\}$.  Each hyperplane $H$ is given by a linear function 
$\alpha(x)$, so that 
	$$H = \{x \in \R^\ell:  \alpha(x) = 0\}.$$  
The choice of $\alpha$ determines a ``positive" and a ``negative" side of $H$.  
The positive side of $H$ is given by 
	$$H^+ = \{x \in \R^\ell: \alpha(x) > 0\}$$
and the negative side of $H$ is given by 
	$$H^- = \{x \in \R^\ell: \alpha(x) < 0\}.$$  
It is clear that an alternative choice of a linear function $\alpha$ for $H$ is 
simply the choice of a positive side, up to positive scalar multiple.  When the 
positive side of a hyperplane $H$ has been specified, we call $H$ an 
\emph{oriented hyperplane}.  

\begin{definition}
A \emph{real hyperplane arrangement} $\A$ in $\R^\ell$ is a finite collection of 
hyperplanes $\{H_1, \ldots, H_n\}$ in $\R^\ell$.  Often, we will refer to a real 
hyperplane arrangement simply as a hyperplane arrangement, unless the context is 
unclear.  
\end{definition}

If the positive side of each hyperplane $H_i$ in $\A$ has been chosen, then we 
call $\A$ an \emph{oriented hyperplane arrangement}.  We say that $\A$ is 
\emph{essential} if 
	$$\bigcap_{H \in \A} H = \{0\}.$$

\begin{example}
Figure \ref{arrange} depicts an essential hyperplane arrangement $\A$ consisting 
of three hyperplanes $H_1$, $H_2$, and $H_3$ in $\R^2$.  The arrows on each 
hyperplane indicate the positive side.
\begin{figure}[h]
\begin{center}

\includegraphics[width=6cm]{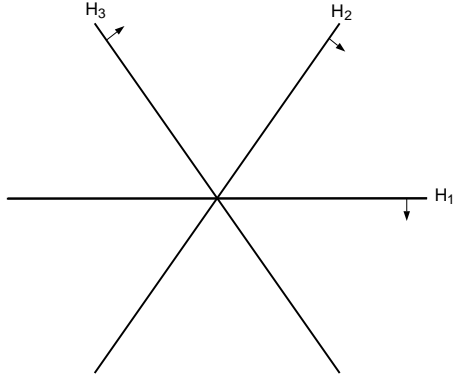}

\caption{An essential hyperplane arrangement}
\label{arrange}
\end{center}
\end{figure}
\end{example} 

The hyperplanes of $\A$ subdivide $\R^\ell$ into \emph{faces}.  We denote the 
set of all faces by $\mathcal{F}(\A)$, which is partially ordered 
via
	$$F' \leq F \ \text{iff} \ F' \subseteq \bar{F},$$
for all $F',F \in \mathcal{F}(\A)$.  We will refer to $\mathcal{F}(\A)$ 
as the face poset for $\A$.  A face $C \in \mathcal{F}(\A)$ is called a 
\emph{chamber} if $C$ is a codimension $0$ face of $\A$.  We denote the set of 
all chambers of $\A$ by $\mathcal{C}(\A)$.  Note that $\mathcal{C}(\A)$ 
consists exactly of the maximal elements of the face poset.  Given a chamber $C$ 
of $\A$, its \emph{walls} are defined to be the hyperplanes $H_i$ of $\A$ such 
that $\bar{C} \cap H_i \neq \emptyset$.  Two chambers $C$ and $D$ are called 
\emph{adjacent} if they have a common wall.

\begin{example}
Let $\A$ be the hyperplane arrangement of Figure \ref{Awfaces1}.  Then 			
	$$\mathcal{C}(\A) = \{C_0, C_1, C_2, C_3, C_4, C_5\}.$$
The walls of $C_0$ are the hyperplanes $H_1$ and $H_2$.  The face poset for $\A$ 
is given in Figure \ref{Awfaces2}.
\begin{figure}[h]
\begin{center}

\includegraphics[width=6cm]{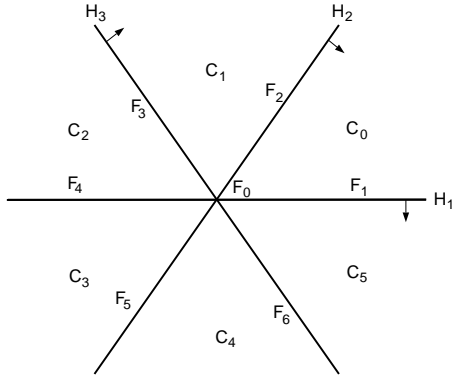}

\caption{Faces of a hyperplane arrangement}
\label{Awfaces1}
\end{center}
\end{figure}

\begin{figure}[h]
\begin{center}

\includegraphics[width=7cm]{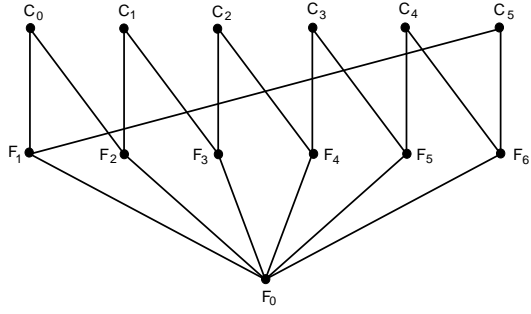}

\caption{Face poset for a hyperplane arrangement}
\label{Awfaces2}
\end{center}
\end{figure}
\end{example}

A \emph{subarrangement} of $\A$ is any subcollection of hyperplanes of $\A$.  
For example, the walls of any chamber $C$ of $\A$ form a subarrangement.  Also, 
for all $F \in \mathcal{F}(\A)$, we can form the subarrangement 
	$$\A_F = \{H \in \A: F \subseteq H\}.$$
Note that $\A_F = \A_{F'}$ iff $|F| = |F'|$, where $|F|$ and $|F'|$ are the 
linear subspaces spanned by $F$ and $F'$, respectively.  Also, note that if $C 
\in \mathcal{C}(\A)$, then $\A_C = \emptyset$.

\bigskip

If $C \in \mathcal{C}(\A)$ and $F \in \mathcal{F}(\A)$, then we denote by $C_F$ 
the unique chamber of $\A_F$ containing $C$.  Note that, since $\A_C = 
\emptyset$ for $C \in \mathcal{C}(\A)$, ${C'}_C = \R^\ell$, for any two chambers 
$C'$ and $C$ of $\A$.

\begin{example}
Figure \ref{subarrange}(b) depicts the subarrangement $\A_F$, where $\A$ is the 
hyperplane arrangement of Figure \ref{subarrange}(a).  Figure 
\ref{subarrange}(b) also shows the unique chamber $C_F$.
\begin{figure}[h]
\begin{center}

\mbox{\subfigure[$\A$]{\includegraphics[width=6cm]{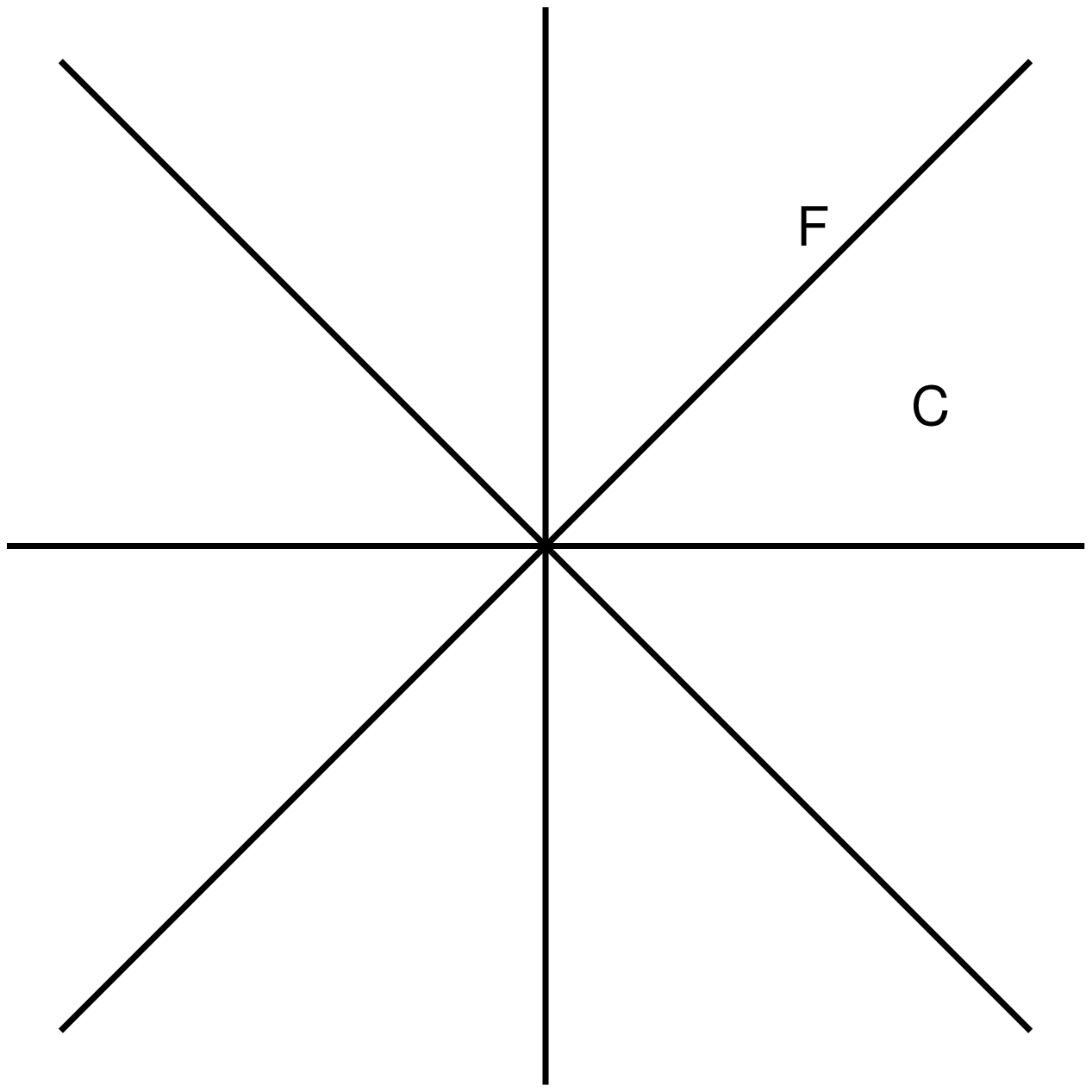}}\quad
\subfigure[$\A_F$]{\includegraphics[width=5.75cm]{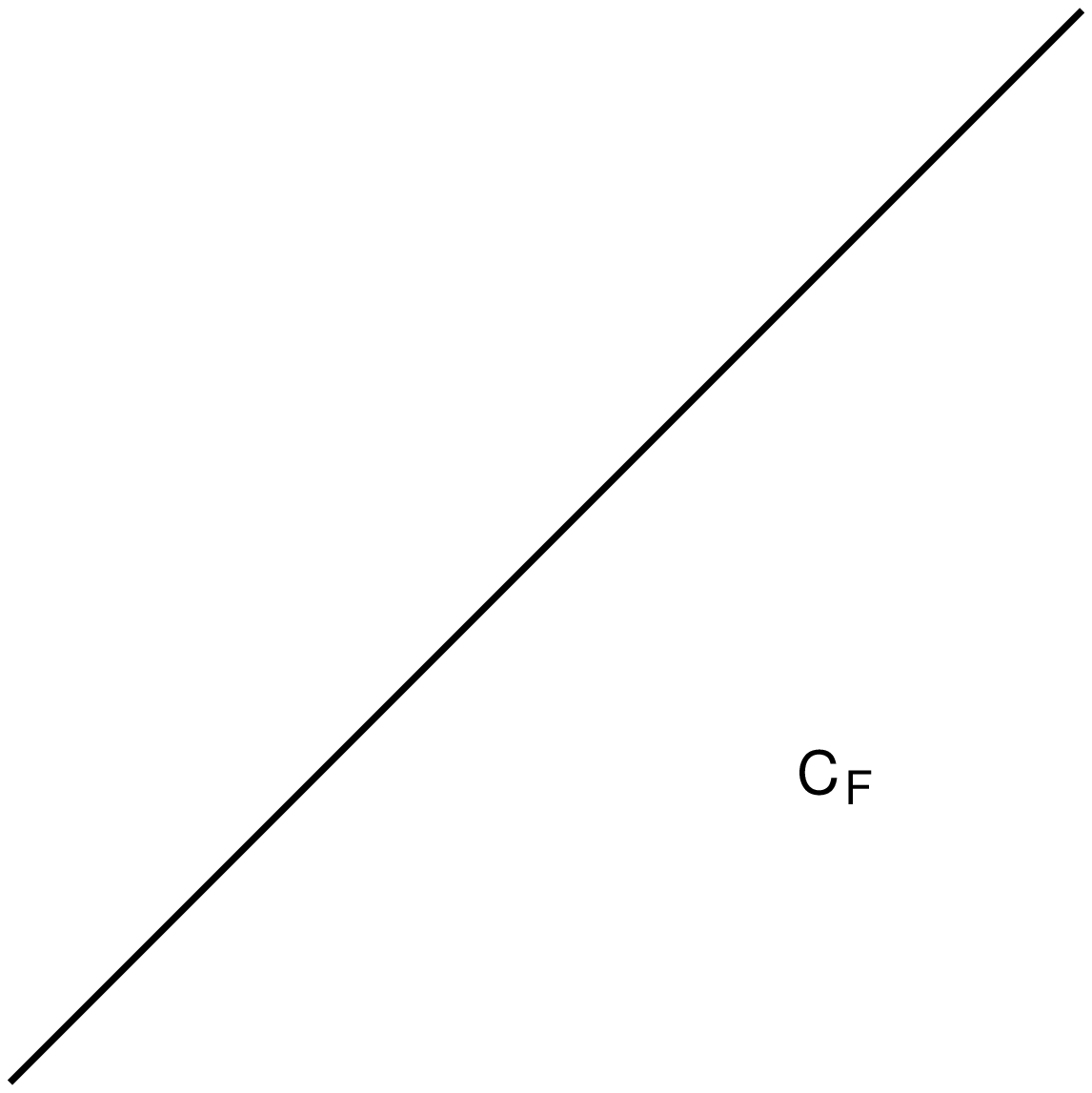}}}

\caption{A subarrangement}
\label{subarrange}
\end{center}
\end{figure}
\end{example}

If $\A$ is a real hyperplane arrangement, then 
	$$\A_\C = \{H+iH: H \in \A\} \subseteq \C^\ell$$
is called the \emph{complexification of} $\A$.  The defining equations for each 
complex hyperplane $H+iH$ are the same as for the real hyperplane $H$, except 
now the independent variables are complex numbers.  Note that a real arrangement 
$\A$ is essential iff its complexification is essential.  Also note that $x+iy 
\in H+iH$ iff $x \in H$ and $y \in H$.  

\bigskip

If $\A$ is a real hyperplane arrangement, then we will denote the 
\emph{complement of the complexification of} $\A$ via 
	$$M(\A) = \C^\ell - \bigcup_{H \in \A} (H+iH).$$
That is, $M(\A)$ consists of all points $x+iy \in \C^\ell$ where both $x$ and 
$y$ are not contained in the same hyperplane $H$.

\bigskip

Every essential arrangement $\A$ in $\R^\ell$ determines a (regular) cellular 
decomposition of  $S^{\ell-1}$, denoted $S_\A$.  A face $F \neq \{0\}$ of 
$\A$ corresponds to the open cell $\Delta(F) = F \cap S^{\ell-1}$, and every 
open cell of this decomposition of $S^{\ell-1}$ has that form.

\bigskip

This cellular decomposition of $S^{\ell-1}$ determines a simplicial 
decomposition of $S^{\ell-1}$, i.e., the barycentric subdivision.  For 
every $F \neq \{0\}$ of $\A$ we fix a point $x(F) \in \Delta (F)$.  A chain 
	$$\{0\} \neq F_0 < \cdots < F_n$$
of faces of $\A$ determines an open simplex 
	$$\phi = x(F_0) \vee \cdots \vee x(F_n)$$
having $x(F_0), \ldots, x(F_n)$ as vertices, and every simplex of this 
simplicial decomposition of $S^{\ell-1}$ has that form.  Let $S_\A'$ denote the 
simplicial decompostion of $S^{\ell-1}$ determined by $\A$. 

\bigskip

The simplicial decomposition of $S^{\ell-1}$ determinies a simplicial 
decomposition of $B^\ell$.  We add the vertex $x(\{0\}) = 0$ to the set of 
vertices for $S^{\ell-1}$.  A chain 
	$$F_0 < \cdots < F_n$$
of faces of $\A$ determines an open simplex 
	$$\phi = x(F_0) \vee \cdots \vee x(F_n)$$
having $x(F_0), \ldots, x(F_n)$ as vertices, and every simplex of this 
simplicial decomposition of $B^\ell$ has that form.  Let $B_\A'$ denote the 
simplicial decomposition of $B^\ell$ determined by $\A$.  

\begin{example}
Let $\A$ be the hyperplane arrangement of Figure \ref{arrange}.  Then the 
simplicial decomposition, $S_\A'$, is shown in Figure \ref{decomp}(a).  By 
adding 
the vertex $x(\{0\})$ to $S_\A'$, we get the simplicial decomposition, $B_\A'$, 
shown in Figure \ref{decomp}(b).
\begin{figure}[h]
\begin{center}

\mbox{\subfigure[$S_\A'$]{\includegraphics[width=5cm]{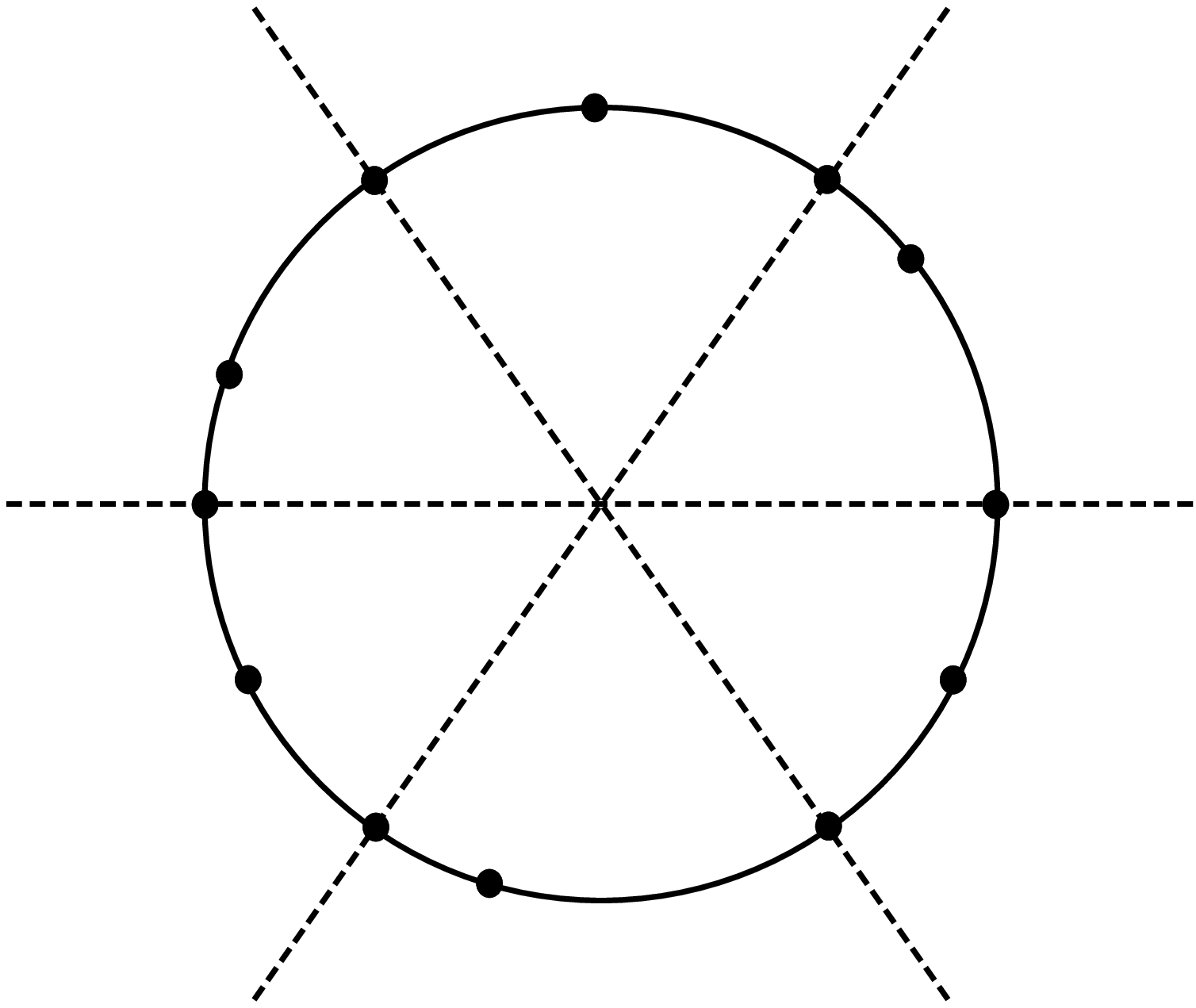}}\quad
\subfigure[$B_\A'$]{\includegraphics[width=5cm]{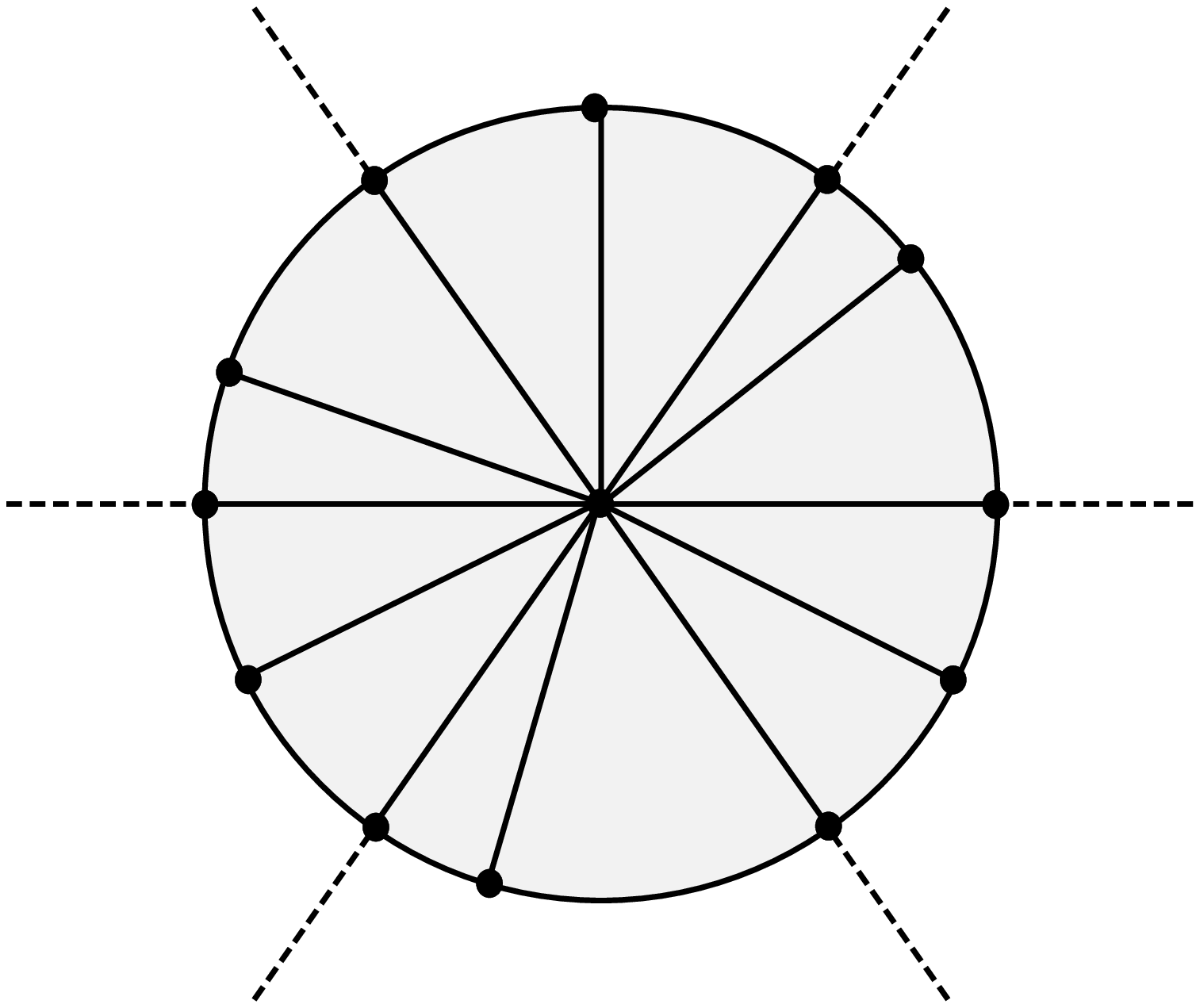}}}

\caption{Simplicial decomposition of $S^1$ and $B^2$}
\label{decomp}
\end{center}
\end{figure} 
\end{example}

Note that, if $F_0 \neq \{0\}$, then 
	$$\phi = x(F_0) \vee \cdots \vee x(F_n) \subseteq S^{\ell-1}.$$  
If $\phi$ is an open cell of $S^{\ell-1}$, then the \emph{cone of} $\phi$ is 
	$$K(\phi) = \{\lambda x: x \in \phi \ \text{and} \ \lambda > 0\}.$$  

\begin{lemma}\label{word}
If 
	$$\phi = x(F_0) \vee \cdots \vee x(F_n)$$
is an open simplex of $S^{\ell-1}$ with 
	$$\{0\} \neq F_0 < \cdots < F_n\text{,}$$
then $K(\phi) \subseteq F_n$.
\end{lemma}

\begin{proof}
Let $y \in K(\phi)$.  Then $y=\lambda_0 x$ for some $x \in \phi$ and some 
$\lambda_0 > 0$.  Since $F_0 \leq \cdots \leq F_n$, $x(F_0), \ldots, x(F_n) \in 
\bar{F_n}$. Then the open convex hull of $x(F_0), \ldots, x(F_n)$ is contained 
in $F_n$, since $F_n$ is convex.  Hence $\phi \subseteq F_n$.  So, $x \in F_n$.  
But $\lambda x \in F_n$ for every $\lambda > 0$, since $F_n$ is a sector.  Hence 
$y \in F_n$.  Therefore, $K(\phi) \subseteq F_n$.  
\end{proof}

Note that the family 
	$$\{K(\phi): \phi \ \text{a simplex of} \ S^{\ell-1}\}$$
is a partition of $\R^\ell-\{0\}$. 
 
\end{chapter}


\begin{chapter}{Arrangements with Group Actions}\label{chap3}

There are many special types of hyperplane arrangements.  In this chapter, we 
will focus our discussion on arrangements which support the action of finite 
reflection groups.  We will begin by introducing some necessary definitions and 
terminology.

\bigskip

Let $H$ be a hyperplane in $\R^\ell$.  The \emph{reflection} with respect to $H$ 
is the linear transformation 
	$$s_H: \R^\ell \to \R^\ell$$
which is the identity on $H$ and is multiplication by $-1$ on the orthogonal 
complement $H^\bot$ of $H$.

\begin{definition}
A \emph{finite reflection group} is a finite group $W$ of orthogonal linear 
transformations of $\R^\ell$ generated by reflections.
\end{definition}

If $W$ is a finite reflection group, then the corresponding hyperplane 
arrangement $\A_W$ consists of all the hyperplanes $H$ where $s_H$ is a 
reflection in $W$.  We will usually suppress the subscript $W$ when referring to 
the hyperplane arrangement corresponding to the finite reflection group $W$.  In 
this case, we will call $\A$ a \emph{reflection arrangement}.

\bigskip

For a complete proof of the following theorem, see \cite{BR} or \cite{HU}.

\begin{theorem}
If $\A$ is a reflection arrangement with reflection group $W$, then $W$ permutes 
the hyperplanes of $\A$.
\end{theorem}

\begin{proof}
One easily verifies that $s_{wH} = ws_Hw^{-1}$ when $s_H \in W$ and $w \in W$.  
This implies that $W$ permutes the hyperplanes of $\A$.
\end{proof}

Since each reflection $s_H \in W$ is linear and every face of $\A$ is an 
intersection of half-spaces and hyperplanes of $\A$, we have the following 
corollary.

\begin{corollary}
If $\A$ is a reflection arrangement with reflection group $W$, then $W$ permutes 
the faces of $\A$.  Furthermore, $W$  preserves the face relation determined by 
$\A$. \hfill $\Box$
\end{corollary}

The following lemma provides us with a useful way of thinking about reflections.

\begin{lemma}
If $s_{H_i}$ and $s_{H_j}$ are reflections over the hyperplanes $H_i$ 
and $H_j$ in $\R^\ell$, respectively.  Then $s_{H_i}s_{H_j} \in W$ and 
$s_{H_i}s_{H_j}$ is a rotation about $H_i \cap H_j$ of twice the angle between 
$H_i$ and $H_j$.  \hfill $\Box$
\end{lemma} 

The requirement that $W$ be finite is very important.  Suppose, for instance, 
that $W$ contains the reflections $s_{H_1}$ and $s_{H_2}$.  Since $W$ is finite, 
the rotation $s_{H_1}s_{H_2}$ has finite order.  Then the angle between $H_1$ 
and $H_2$ must be a rational multiple of $\pi$.  In higher dimensions, the 
restrictions on the angles between any two hyperplanes is even more severe.  In 
general, we have the following result.

\begin{proposition}\label{condition}
Let $\A$ be an essential reflection arrangement in $\R^\ell$ with reflection 
group $W$.  When $\ell=2$, $W$ is a finite reflection group iff the angle 
between any pair of hyperplanes $H_i$ and $H_j$ of $\A$ is a rational multiple 
of $\pi$. When $\ell>2$, $W$ is an irreducible finite reflection group only if 
the angle between any pair of hyperplanes $H_i$ and $H_j$ of $\A$ is $\pi/2$, 
$\pi/3$, $\pi/4$, $\pi/5$, or $\pi/6$. \hfill $\Box$
\end{proposition}

Here, irreducible means irreducible as representations of abstract groups.  
Finite reflection groups have been completely classified up to isomorphism.  
Below, we list the irreducible finite reflection groups briefly with little 
explanation.  Also, the hyperplane arrangements associated to each reflection 
group are assumed to be essential.  All other finite reflection groups can be 
obtained from these by taking direct sums and, possibly, adding an extra summand 
on which the group acts trivially.

\begin{itemize}

\item  \emph{Type $A_n$} ($n \geq 1$):  Here $W$ is the symmetric 
group on $n+1$ letters, acting on a certain $n$-dimensional subspace of 
$\R^{n+1}$ by permuting the coordinates.  This group is the group of symmetries 
of a regular $n$-simplex, and it can also be described as the Weyl group of the 
root system of type $A_n$.

\item  \emph{Type $B_n$} ($n \geq 2$):  This is the group $W$ of 
signed permutations acting on $\R^n$.  The group $W$ is the group of 
symmetries of the $n$-cube $[-1,1]^n$; it is also the Weyl group of the root 
system of type $B_n$.  

\item  \emph{Type $D_n$} ($n \geq 4$):  This group $W$ is the Weyl group of 
a root system, called the root system of type $D_n$, but it does not 
correspond to any regular solid.  It is the group of signed permutations with an 
even number of sign changes.  It also happens to be a subgroup of index $2$ 
of the reflection group of type $B_n$.

\item  \emph{Type $E_n$} ($n = 6,7,8$):  This is the Weyl group of the 
root system of the same name.  It does not correspond to any regular solid.

\item  \emph{Type $F_4$}:  This is the Weyl group of the root system of the 
same name; it is also the group of symmetries of a certain self-dual $24$-sided 
regular solid in $\R^4$ whose 1-codimensional faces are solid octahedra.

\item  \emph{Type $H_n$} ($n = 3,4$):  This does not correspond to any 
root system, but it is the symmetry group of a regular solid $X$.  When $n = 
3$, $X$ is the dodecahedron or, dually, the icosahedron.  When $n = 4$, $X$ 
is a $120$-sided solid in $\R^4$ with dodecahedral faces or, dually, a 
$600$-sided solid with tetrahedral faces.  

\item  \emph{Type $I_2(m)$} ($m = 5$ or $m \geq 7$):  The group $W$ is the 
dihedral group of order $2m$.  It is the symmetry group of a regular $m$-gon, 
but it does not correspond to any root system.

\end{itemize}

\begin{example}
Figure \ref{square} shows the square with its lines of symmetry.  The four 
lines of symmetry form the hyperplane arrangement $\A$, which has the 
corresponding finite reflection group of type $B_2$.  The finite reflection 
group consists of all distinct products of the reflections $s_{H_1}, s_{H_2}, 
s_{H_3}$, and $s_{H_4}$.  Notice that for any pair of chambers of $\A$, there 
exists a product of reflections which maps one chamber to the other.  That is, 
$B_2$ acts transitively on the chambers of $\A$.
\begin{figure}[h]
\begin{center}

\includegraphics[width=6cm]{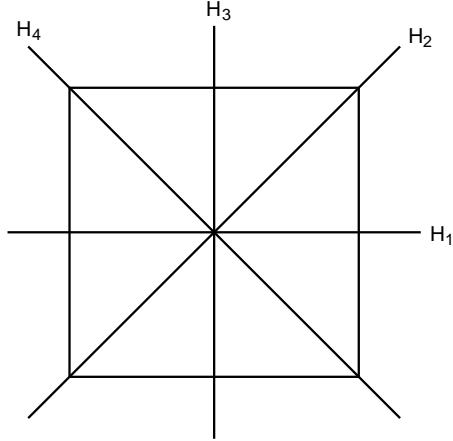}

\caption{A finite reflection arrangement of type $B_2$}
\label{square}
\end{center}
\end{figure}
\end{example}

Now, assume that $\A$ is an essential reflection arrangement in $\R^\ell$ with 
reflection group $W$.  Anytime we fix a chamber $C_0$ of $\A$, we will refer to 
$C_0$ as the \emph{base chamber} for $\A$.  Using this notion of base chamber, 
we will list a series of facts about finite reflection groups without proof.

\begin{lemma}\label{L1}
Let $\A$ be an essential reflection arrangement in $\R^\ell$ with reflection 
group $W$.  Suppose that $C_0$ is the base chamber of $\A$ with walls $H_1, 
\ldots, H_n$.  Then $W$ is generated by the reflections $s_{H_1}, \ldots, 
s_{H_n}$. \hfill $\Box$
\end{lemma}

\begin{lemma}\label{L2} 
Let $\A$ be an essential reflection arrangement in $\R^\ell$ with corresponding 
reflection group $W$.  Then the action of $W$ is transitive and free on the set 
of chambers.  In particular, the number of chambers is equal to the order of 
$W$.  \hfill $\Box$
\end{lemma}

Recall that an action is \emph{free} if point stabilizers are trivial.  To prove 
that $W$ acts freely, one shows that any $s \in W$ which fixes a closed 
chamber $\bar{C}$ setwise, also fixes $\bar{C}$ pointwise, and hence must be the 
identity. 

\bigskip

Lemma \ref{L2} also implies that there is a one-to-one correspondence between 
the chambers of $\A$ and the elements of $W$.  That is, we can identify chambers 
with elements of $W$.  If $C_0$ is the base chamber of $\A$, then identify $C_0$ 
with the identity element of $W$.  Then for a chamber $C \neq C_0$ of $\A$, 
identify $C$ with the unique element $w \in W$ such that $wC_0 = C$.  The 
element $w$ exists because $W$ is transitive and is unique since $W$ acts freely 
on chambers.  Note that we can write $w$ as a product of reflections $s_{H_1}, 
\ldots, s_{H_n}$, where $H_1, \ldots, H_n$ are the walls of $C_0$, according to 
Lemma \ref{L1}.  

\begin{lemma}\label{L4}
Let $\A$ be an essential reflection arrangement in $\R^\ell$ with reflection 
group $W$.  Suppose that $C_0$ is the base chamber of $\A$ with walls $H_1, 
\ldots, H_n$, so that $S = \{s_{H_1}, \ldots, s_{H_n}\}$ is the set of 
reflections with respect to the walls of $C_0$.  Then $\bar{C_0}$ is a set of 
representatives for the $W$-orbits in $\R^\ell$.  Moreover, the stabilizer $W_x$ 
of a point $x \in \bar{C_0}$ is the subgroup generated by $S_x = \{s \in S: 
sx=x\}$.  In particular, $W_x$ fixes every point of $\bar{F}$, where $F \leq 
C_0$ is the smallest face containing $x$.  \hfill $\Box$
\end{lemma}

\begin{lemma}\label{L5}
Let $\A$ be an essential reflection arrangement in $\R^\ell$ with reflection 
group $W$.  Suppose that $C_0$ is the base chamber of $\A$ with walls $H_1, 
\ldots, H_n$.  Then each face $F$ of $\A$ is congruent to a unique face $F_0$ of 
$C_0$ under the action of $W$.  Furthermore, if $w \in W$ with $wC_0 = C$ and 
$F_0 \leq C_0$, then $wF_0 = F \leq C$.  \hfill $\Box$
\end{lemma}

Using the above series of lemmas, we can show that the simplicial decomposition 
of $S^{\ell-1}$ determined by $\A$ can be constructed to be invariant under the 
action of $W$.  

\begin{theorem}[{\rm Salvetti} \cite{SA1}]
If $\A$ is an essential reflection arrangement in $\R^\ell$ with reflection 
group $W$, then we can construct $S_\A'$ to be $W$-invariant.
\end{theorem}

\begin{proof}
Let $\A$ be an essential reflection arrangement in $\R^\ell$ with reflection 
group $W$.  First, fix a base chamber $C_0$ of $\A$.  Suppose that $F_{01}, 
\ldots, F_{0n}$ are the nonzero faces of $C_0$.  Now, pick a point $x(C_0) \in 
C_0 \cap S^{\ell-1}$ and a point $x(F_{0j}) \in F_{0j} \cap S^{\ell-1}$ for each 
$j$.  For each chamber $C_i \neq C_0$ of $\A$ with nonzero faces $F_{i1}, 
\ldots, F_{in}$, suppose that $s_i \in W$ is identified with $C_i$.  That is, 
	$$C_i = s_iC_0$$
for each $i$.  We can choose $s_i$ uniquely according to the remark following 
Lemma \ref{L2}.  Note that we know that each $C_i$ has the same number of faces 
as $C_0$ by Lemma \ref{L5}.  Now, set 
	$$x(C_i) = s_ix(C_0)$$
for each $i$.  Also by Lemma \ref{L5}, each face of $C_i$ is congruent to a 
unique face of $C_0$.  Without loss of generality, assume that $F_{ij}$ is 
congruent to $F_{0j}$ for each $j$.  Then we can set 
	$$x(F_{ij}) = s_ix(F_{0j})$$
for each $j$, according to Lemma \ref{L5}.  It is clear that each 
	$$x(F_{ij}) \in F_{ij} \cap S^{\ell-1}.$$  
Now, for every chain of faces 
	$$F_1 < \cdots < F_n\text{,}$$
form the simplex 
	$$x(F_1) \vee \cdots \vee x(F_n)$$
on $S^{\ell-1}$.  It is clear that $W$ is invariant on simplices of this form 
since each $s \in W$ is linear and $W$ preserves order.  With the simplicial 
decomposition of $S^{\ell-1}$ constructed in this way, it is 
clear that $S_\A'$ is $W$-invariant.  
\end{proof}

Note that if we include the point $x(\{0\})$ in our decomposition, then we can 
also construct $B_\A'$ to be $W$-invariant.

\begin{example}
Consider the reflection arrangement $\A$ given in Figure \ref{Wdecomp1}.  Let 
$W$ be the reflection group of $\A$ and let $C_0$ be the base chamber. For the 
faces $F_1$ and $F_2$ of $C_0$ we fix the points $x(F_1) \in F_1 \cap S^1$ and 
$x(F_2) \in F_2 \cap S^1$.  See Figure \ref{Wdecomp2}.  In this case, we don't 
have much choice as to what $x(F_1)$ and $x(F_2)$ are.  Next, pick $x(C_0) \in 
C_0 \cap S^1$.  The rest of the vertices of $S_\A'$ are chosen by looking at the 
orbits of the points $x(F_1), x(F_2)$, and $x(C_0)$, under combinations of 
reflections over the walls of $C_0$.  It should be clear from Figure 
\ref{Wdecomp2} that the indicated points are what we get.  The resulting 
simplicial decomposition of $S^1$ will be invariant under the action of $W$.   
\begin{figure}[h]
\begin{center}

\includegraphics[width=6cm]{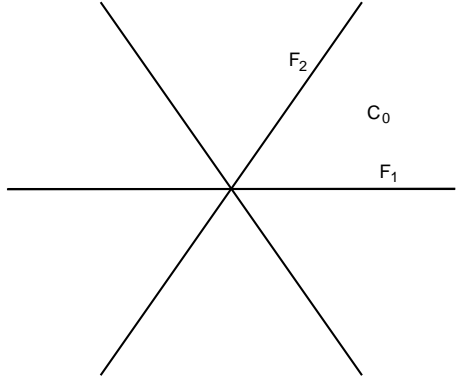}

\caption{A hyperplane arrangement}
\label{Wdecomp1}
\end{center}
\end{figure}

\begin{figure}[h]
\begin{center}

\includegraphics[width=6cm]{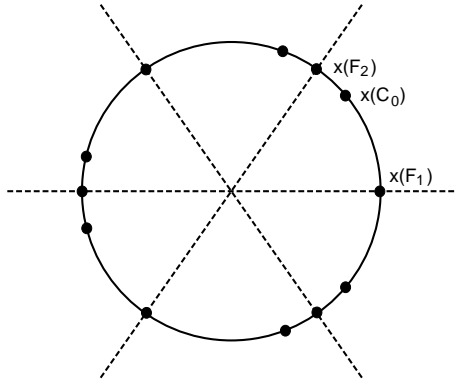}

\caption{$W$-invariant $S_\A'$}
\label{Wdecomp2}
\end{center}
\end{figure}
\end{example}

\begin{paragraph}{The Braid Arrangement}

One class of reflection arrangements which are of particular interest to us are 
called the \emph{braid arrangements}.  We denote the braid arrangement in 
$\R^\ell$ by $\A_{\ell-1}$.  For each dimension $\ell$, we define
	$$\A_{\ell-1} = \{H_{ij}: 1 \leq i < j \leq \ell\} \subseteq \R^\ell$$
where $H_{ij} = \{x \in \R^\ell: x_i = x_j\}$.  We orient each $H_{ij}$, so that 
$x_i > x_j$ defines the positive side of $H_{ij}$ for $i<j$.  So, each $H_{ij}$ 
is positive on the side containing the positive $x_i$ axis.

\bigskip

Clearly, each $\A_{\ell-1}$ is not essential, since 
	$$\bigcap \A_{\ell-1} = \{x \in \R^\ell: x_1 = \cdots = x_\ell\} \neq 
\{0\}.$$  
However, we can ``make" $\A_{\ell-1}$ essential and still preserve the face 
poset for $\A_{\ell-1}$.  Take the plane defined by
	$$x_1 + \cdots + x_\ell=0$$
perpendicular to the line 
	$$\{x_1 = \cdots = x_\ell\}$$
and then intersect the hyperplanes of $\A_{\ell-1}$ with that plane.  The ``new" 
arrangement will be isomorphic to an arrangement in $\R^{\ell-1}$ and will 
preserve the structure of the face poset for $\A_{\ell-1}$.

\begin{example}
Figure \ref{braid} represents the braid arrangement $\A_2$.  Note that this is 
not $\A_2$, but rather a projection of $\A_2$ onto the plane defined by
	$$x_1+x_2+x_3=0$$
perpendicular to the line 
	$$\{x_1 = x_2 = x_3 \}.$$
We will become more familiar with this particular arrangement as we continue.
\begin{figure}[h]
\begin{center}

\includegraphics[width=6cm]{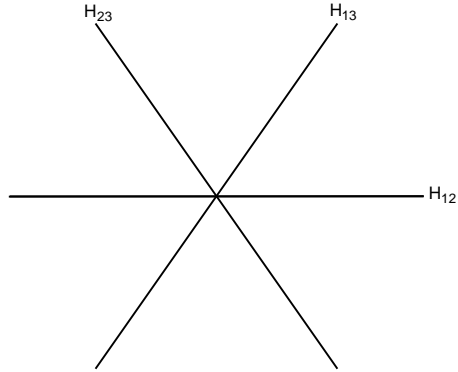}

\caption{The braid arrangement, $\A_2$.}
\label{braid}
\end{center}
\end{figure}
\end{example} 
 
Note that $\A_{\ell-1}$ has $\ell\,!$ chambers and that $\A_{\ell-1}$ consists 
of $\binom{\ell}{2}$ hyperplanes.  This implies that if $W$ is the finite 
reflection group associated with $\A_{\ell-1}$, then the order of $W$ is 
$\ell\,!$.  In fact, we have the following result, which we will state without 
proof.

\begin{proposition}\label{P1}
If $W$ is the finite reflection group associated to $\A_{\ell-1}$, then $W$ is 
of type  $A_{\ell-1}$.  That is, $W$ is isomorphic to the symmetric group on 
$\ell$ letters. \hfill $\Box$
\end{proposition}

If $H_{ij}$ is a hyperplane of $\A_{\ell-1}$, then the reflection $s_{H_{ij}}$ 
acts on $x \in \R^\ell$ by permuting the $i^{th}$ and $j^{th}$ components of 
$x$.  Thus, we can think of $s_{H_{ij}}$ as a transposition of the symmetric 
group.  We will make no distinction between the group of reflections of 
$\A_{\ell-1}$ and $S_\ell$.  

\bigskip

By Proposition \ref{P1}, we can identify the chambers of $\A_{\ell-1}$ with the 
permutations of the symmetric group.  In fact, if $C_0$ is the base chamber of 
$\A_{\ell-1}$ consisting of walls $H_1, \ldots, H_n$, then we can identify each 
chamber of $\A$ with a product consisting of the transpositions $s_{H_1}, 
\ldots, s_{H_n}$.

\bigskip

It is clear that the complement of the braid arrangement, $M(\A_{\ell-1})$, 
consists of all vectors $z=(z_1, \ldots, z_\ell) \in \C^\ell$ such that $z_1 
\neq \cdots \neq z_\ell$.  In other words, $M(\A_{\ell-1})$ can be thought of as 
the set of ordered $\ell$-tuples of distinct complex numbers.  To make this more 
precise we will introduce the following definition.

\begin{definition}
Define $\hat{F_\ell}(\C)$ to be the space of labelled configurations of $\ell$ 
distinct points in $\C$.
\end{definition}

\begin{example}
Figure \ref{labelled} shows two examples of labelled configurations of 
$\hat{F}_3(\C)$.  In each example the configuration of points is the same, 
but each is labelled differently, making each a unique element of 
$\hat{F}_3(\C)$.
\begin{figure}[h]
\begin{center}

\includegraphics[width=4cm]{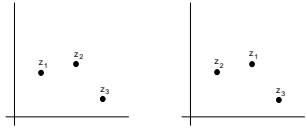}

\caption{Labelled point configurations of $\hat{F}_3(\C)$.}
\label{labelled}
\end{center}
\end{figure}
\end{example}

A point $z=(z_1, \dots, z_\ell) \in M(\A_{\ell-1})$ determines the labelled 
configuration of points consisting of $z_1, \ldots, z_\ell$.  Also, each 
labelled configuration of $\hat{F_\ell}(\C)$ clearly determines a point $z \in 
M(\A_{\ell-1})$.  Hence, we have the following result.

\begin{proposition}
The complement of the braid arrangement, $M(\A_{\ell-1})$, is homeomorphic to 
$\hat{F_\ell}(\C)$.  \hfill $\Box$
\end{proposition}

\end{paragraph}  

\end{chapter}


\begin{chapter}{Oriented Matroids}

The notion of oriented matroid is a natural mathematical concept which presents 
itself in many different forms.  Oriented matroids can thought of as discrete 
combinatorial abstractions of point configurations over the reals, of convex 
polytopes, of directed graphs, and of oriented hyperplane arrangements.  We will 
restrict our discussion of oriented matroids to the context of oriented 
hyperplane arrangements.

\bigskip

A characteristic feature of oriented matroids is the variety of different but 
equivalent axiom systems.  No single axiom system takes priority over the 
others.  We will describe oriented matroids from a covector point of view.  But 
before we describe oriented matroids and how they arise from oriented hyperplane 
arrangements, we will first introduce some necessary ideas concerning covectors 
and then list the covector axioms of an oriented matroid.

\bigskip

For our purposes, a \emph{covector} $X$ is simply an element of $\{+,-,0\}^n$.  
That is, a covector $X$ is an $n$-tuple where each component $X_i$ is a $+$, 
$-$, or $0$.  We will call a covector $T$ a \emph{tope} if $T_i \neq 0$ for all 
$i$.  The \emph{opposite} of a covector $X$ is $-X$ defined componentwise via  
\begin{equation*}
(-X)_i =
	\begin{cases}
	-&  \text{if $X_i = +$},\\
	+&  \text{if $X_i = -$},\\
	0&  \text{if $X_i = 0$}.
	\end{cases}
\end{equation*}
For example, if $X = ++0-$, then $-X = --0+$.  The \emph{composition} of two 
covectors $X$ and $Y$ is $X \circ Y$ defined componentwise via
\begin{equation*}
(X \circ Y)_i =
	\begin{cases}
	X_i&  \text{if $X_i \neq 0$},\\
	Y_i&  \text{if $X_i = 0$}.
	\end{cases}
\end{equation*}
Note that this composition is associative, but not commutative.  The 
\emph{separation set} of covectors $X$ and $Y$ is $S(X,Y) =\{i:  X_i = (-Y)_i 
\neq 0\}$.  Now, we are ready to list the covector axioms for an oriented 
matroid.

\begin{definition}
A set $\mathcal{L} \subseteq \{+,-,0\}^n$ is the set of covectors of an oriented 
matroid iff $\mathcal{L}$ satisfies:
\begin{itemize}
\item[(L0)]  $0 \in \mathcal{L}$,
\item[(L1)]  $X \in \mathcal{L}$ implies $-X \in \mathcal{L}$,
\item[(L2)]  $X,Y \in \mathcal{L}$ implies $X \circ Y \in \mathcal{L}$,
\item[(L3)]  if $X,Y \in \mathcal{L}$ and $i \in S(X,Y)$, then there exists $Z 	
\in \mathcal{L}$ such that $Z_i = 0$ and $Z_j = (X \circ Y)_j = (Y \circ X)_j$ 
for all $j \notin S(X,Y)$.
\end{itemize}
\end{definition}

We will often refer to a set of covectors $\mathcal{L}$ for an oriented matroid 
as the oriented matroid itself. 

\bigskip

The oriented matroid of a hyperplane arrangement arises in a very 
natural way.  Let $\A$ be a hyperplane arrangement in $\R^\ell$ 
consisting of hyperplanes $\{H_1, \ldots, H_n\}$.  To each face $F$ of $\A$ 
corresponds a covector $X$.  In this case, we will say that $X$ \emph{labels} 
$F$.  The oriented matroid associated with $\A$ consists of exactly these 
covectors.  We describe this correspondence below.  

\bigskip

Recall that each hyperplane $H_i$ is given by a linear function $\alpha_i$, so 
that 
	$$H_i=\{x \in \R^\ell: \alpha_i(x)=0\}$$
and 
	$${H_i}^+ =\{x \in \R^\ell: \alpha_i(x) > 0\}.$$  
The linear functions $\alpha_1, \ldots, \alpha_n$ determine the covectors 
corresponding to the faces of $\A$.  In fact, $\alpha_i$ determines the $i^{th}$ 
component of each covector.  Let $F$ be a face of $\A$ and let $x \in F$.  Then 
the covector $X$ corresponding to $F$ is given by
\begin{equation*}
X_i =
	\begin{cases}
	+&  \text{if $\alpha_i(x) > 0$},\\
	-&  \text{if $\alpha_i(x) < 0$},\\
	0&  \text{if $\alpha_i(x) = 0$}.
	\end{cases}
\end{equation*}

\bigskip

In the above definition, $X_i$ is independent of the point $x \in F$.  
However, $X_i$ is dependent upon the numbering of the hyperplanes.  A different 
numbering of the hyperplanes will give rise to a different set of covectors.  
Yet the set of corresponding covectors will be in one-to-one correspondence with 
the faces of $\A$, regardless of the numbering.  Let $\mathcal{L}(\A)$ denote 
the set of covectors for the oriented hyperplane arrangement $\A$.  Likewise, 
let $\mathcal{T}(\A)$ denote the set of topes of $\A$.  Note that the topes of 
$\A$ label exactly the chambers of $\A$, since each $\alpha_i$ is nonzero on a 
chamber.

\bigskip

There is a very convenient geometric interpretation of the composition of 
covectors for hyperplane arrangements.  Let $\mathcal{L}(\A)$ be the set of 
covectors for an oriented hyperplane arrangement $\A$.  Let $F$ and $F'$ be 
faces of $\A$.  Assume that the covectors $X$ and $X'$ label $F$ and $F'$, 
respectively.  Then we can interpret the covector composition of $X \circ X'$ 
geometrically as ``stand on $F$ and take one step towards $F'$."  That 
is, pick an arbitrary point $x \in F$, and then move a small distance $\epsilon$ 
from $x$ towards $F'$.  The face that you end up in is labelled by $X \circ X'$. 
  
\begin{theorem}\label{matroid}
If $\A$ is an oriented essential hyperplane arrangement, then $\mathcal{L}(\A)$ 
is the set of covectors for an oriented matroid.  That is, $\mathcal{L}(\A)$ is 
the oriented matroid for $\A$.
\end{theorem}

\begin{proof}
Let $\mathcal{L}(\A)$ be the set of covectors for an oriented hyperplane 
arrangement $\A$.  We must show that $\mathcal{L}(\A)$ satisfies (L0)-(L3).  
Since $\{0\}$ is a face of $\A$, (L0) is clearly satisfied.  To see that (L1) is 
satisfied, assume that $X$ is in $\mathcal{L}(\A)$ and labels the face $F$ of 
$\A$.  Then $-X$ labels the face of $\A$ opposite $F$.  So, $-X$ is in 
$\mathcal{L}(\A)$ and hence (L1) is satisfied.  Also, (L2) is clearly satisfied, 
based upon our discussion of the geometric interpretation of covector 
composition.  Now, let $X,Y \in \mathcal{L}(\A)$ and let $i \in S(X,Y)$.  Assume 
that $X$ and $Y$ label the faces $F$ and $F'$, respectively.  Geometrically, $i 
\in S(X,Y)$ implies that the hyperplane $H_i$ separates the faces $F$ and $F'$.  
By separates, we mean that for any line segment $L$ joining any two arbitrary 
points $x(F)$ and $x(F')$ in $F$ and $F'$, respectively, $L$ passes through 
$H_i$.  Fix such a line segment $L$.  Let $G$ be the face of $H_i$ that $L$ 
passes through and let $Z$ be the covector that labels $G$.  Since $Z$ lies on 
$H_i$, $Z_i = 0$.  Now, let $j \notin S(X,Y)$.  Then $H_j$ does not separate $X$ 
and $Y$.  Clearly, the faces labelled by $X \circ Y$, $Y \circ X$, and $Z$ are 
all on the same side of $H_j$.  Thus, 
	$$Z_j = (X \circ Y)_j = (Y \circ X)_j.$$
So, (L3) is also satisfied. 
\end{proof}

\begin{example}\label{ex4.1}
Let $\A_2$ be oriented as in Figure \ref{covectors}.  Given the orientation on 
$\A_2$, the faces of $\A_2$ are labelled by their corresponding covectors.
\begin{figure}[h]
\begin{center}

\includegraphics[width=6cm]{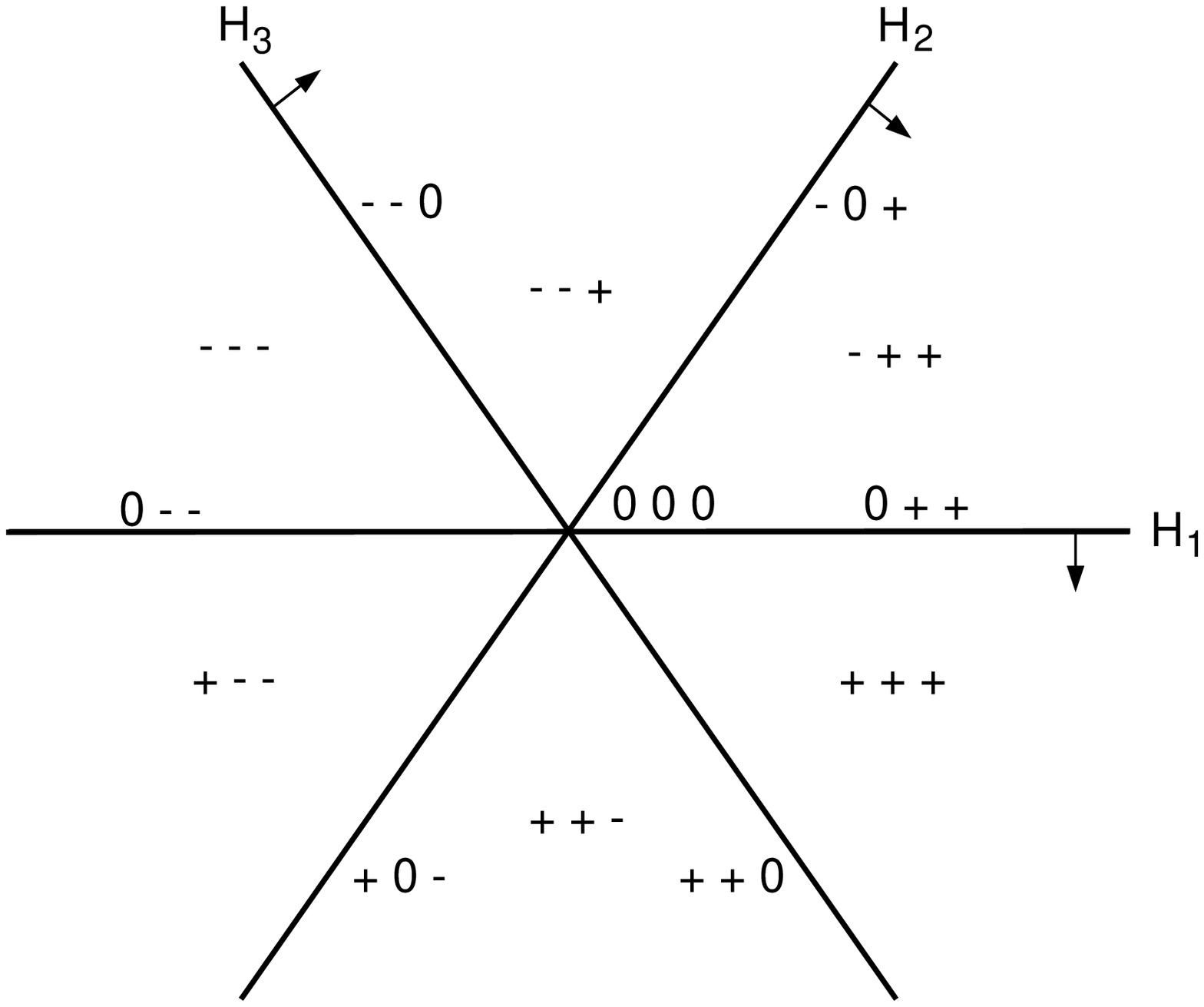}

\caption{Covectors of $\A_2$}
\label{covectors}
\end{center}
\end{figure}
\end{example}

We define the \emph{rank} of an oriented matroid $\mathcal{L}$ to be the length 
of the maximal chain in the poset of covectors.  If $\mathcal{L}(\A)$ is the 
oriented matroid of a hyperplane arrangement $\A$, then the rank of 
$\mathcal{L}(\A)$ is the codimension of $\bigcap_{H \in \A}H$.

\bigskip

If $\mathcal{L}$ is an oriented matroid, then we say that $\mathcal{L}$ is 
\emph{loop-free} if for every $i$, there exists at least one covector where the 
$i^{th}$ component is nonzero.  Clearly, if $\A$ is a hyperplane arrangement, 
then $\mathcal{L}(\A)$ is loop-free, since there is at least one ``nonzero" face 
on either side of each hyperplane in $\A$. 

\bigskip

Now, define the partial order ``$\leq$" on the set $\{+,-,0\}$ via $0<+$, $0<-$,  
with $+$ and $-$ incomparable.  See Figure \ref{po}.
\begin{figure}[h]
\begin{center}

\includegraphics[width=4cm]{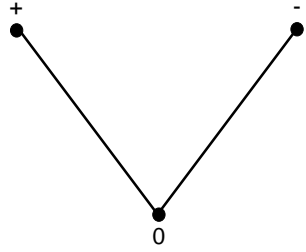}

\caption{The partial ordering on $\{+,-,0\}$}
\label{po}
\end{center}
\end{figure}
This induces a partial order on covectors in which covectors are compared 
componentwise.  That is, for covectors $X$ and $X'$
	$$X \leq X' \ \text{iff} \ X_i \leq {X'}_i\text{,}$$
for all $i$.  It is easy to check that this ordering on covectors is consistent 
with the ordering on faces.  That is, if $X$ and $X'$ are the covectors 
corresponding to $F$ and $F'$, respectively, 
	$$X \leq X' \ \text{iff} \ F \leq F'.$$  
Thus, $\mathcal{F}(\A)$ and $\mathcal{L}(\A)$ are isomorphic as posets.  

\begin{example}
Let $\A_2$ be oriented as in Example \ref{ex4.1}.  Then the covectors of $\A_2$, 
given this orientation, determine the poset for $\mathcal{L}(\A_2)$ in Figure 
\ref{coposet}. 
\begin{figure}[h]
\begin{center}

\includegraphics[width=8cm]{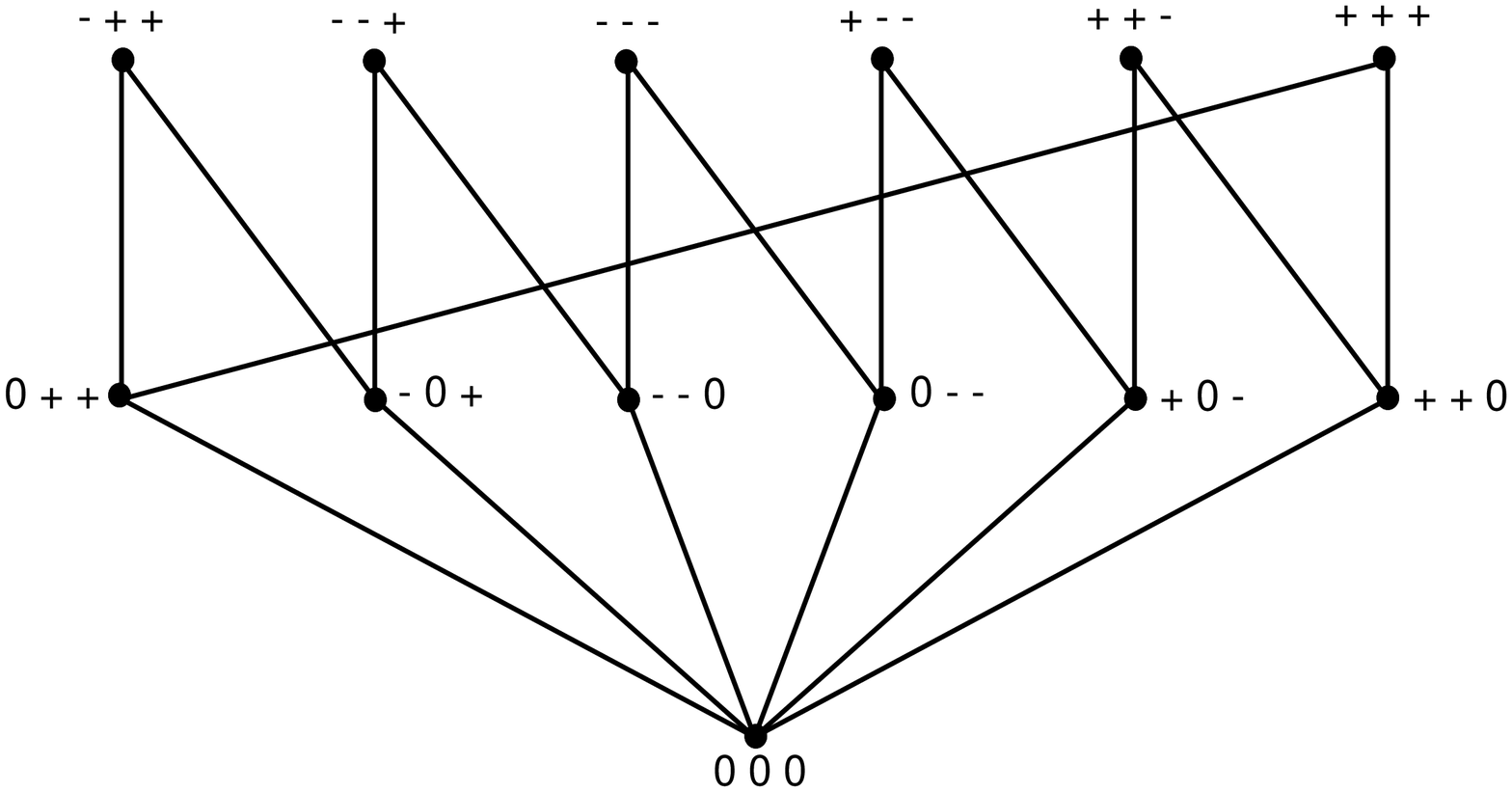}

\caption{Poset of the oriented matroid for $\A_2$}
\label{coposet}
\end{center}
\end{figure}
\end{example}

There is an almost automatic identification between the covectors of the 
oriented matroid of a hyperplane arrangement $\A$ and the faces of $\A$.  Since 
$\mathcal{F}(\A)$ and $\mathcal{L}(\A)$ are isomorphic as posets, it will be 
very convenient for us to use $F$ to refer to both the face of $\A$ and the 
covector which labels $F$.  Similarly, we will use $C$ to denote both a chamber 
of $\A$ as well as the tope that labels $C$.  We will stick to this convention 
for the remainder of the paper.

\bigskip

\begin{paragraph}{Arrangements of Pseudospheres}
Not every oriented matroid $\mathcal{L}$ arises from a hyperplane arrangement 
$\A$.  However, every oriented matroid can be realized as an arrangement of 
oriented pseudospheres.  A \emph{pseudosphere} $S$ is a locally flat 
PL-submanifold of $S^{\ell-1}$ which is homeomorphic to $S^{\ell-2}$.  $S$ 
divides $S^{\ell-1}$ into two components, $S^+$ and $S^-$, called 
\emph{pseudohemispheres}, each of which is homeomorphic to $B^{\ell-1}$.  

\bigskip

We define an \emph{arrangement of pseudospheres} $\mathcal{S}$ to be a finite 
set of pseudospheres $\{S_1, \ldots, S_n\}$ in $S^{\ell-1}$ such that
\begin{itemize}
\item[(S1)]  Every non-empty intersection $S_J = \bigcap_{i \in J}S_i$ is 
homeomorphic to a sphere of some dimension for every $J \subseteq \{1, \ldots, 
n\}$. 
\item[(S2)]  For every non-empty intersection $S_J$ and every $i \notin \{1, 
\ldots, n\}$, the intersection $S_J \cap S_i$ is a pseudosphere in $S_J$ with 
sides $S_J \cap {S_i}^+$ and $S_J \cap {S_i}^-$.   
\end{itemize}
If $S_{\{1, \ldots, n\}} = \emptyset$, then $\mathcal{S}$ is called 
\emph{essential}.  If a positive side and negative side of each pseudosphere has 
been chosen, then we say that $\mathcal{S}$ is an \emph{oriented arrangement of 
pseudospheres}.  That is, if ${S_i}^+$ and ${S_i}^-$ have been specified for 
each pseudosphere $S_i$, then $\mathcal{S}$ has been oriented.  By saying that 
an arrangement of pseudospheres $\mathcal{S}$ is \emph{centrally symmetric} we 
mean that each pseudosphere $S_i$ of $\mathcal{S}$ is invariant under the 
antipodal mapping.

\bigskip

We form the oriented matroid for an arrangement of pseudospheres 
$\mathcal{L}(\mathcal{S})$ in much the same way that we formed the oriented 
matroid for an arrangement of hyperplanes.  The pseudospheres subdivide 
$S^{\ell-1}$ into cells.  Let $\mathcal{F}(\mathcal{S})$ denote the set of cells 
of $\mathcal{S}$.  Each $F \in \mathcal{F}(\mathcal{S})$ can be labeled by a 
covector determined by the pseudohemispheres that $F$ lies in.  We are now 
prepared to understand the statement of the following theorem, called the 
Topological Representation Theorem, which is concerned with the realizability of 
an oriented matroid.  For a proof, see \cite{BJ}.

\begin{theorem}
Let $\mathcal{L} \subseteq \{+,-,0\}^n$.  Then $\mathcal{L}$ is the set of 
covectors of a loop-free oriented matroid of rank $\ell$ iff $\mathcal{L}$ is 
isomorphic to the face poset $\mathcal{F}(\mathcal{S})$ for some oriented 
arrangement of pseudospheres $\mathcal{S}$ in $S^{\ell-1}$, which is essential 
and centrally symmetric. \hfill $\Box$
\end{theorem}

Note that if $\A$ is a hyperplane arrangement, then $\mathcal{F}(\A)$ is 
independent of both the numbering and the orientation of the hyperplanes of 
$\A$.  However, if we renumber or reorient the hyperplanes of $\A$, then we 
will end up with a different set of covectors for $\mathcal{L}(\A)$.  But sets 
of covectors which arise from renumbering or reorienting the hyperplanes of $\A$ 
are isomorphic as posets. This is clear from the Topological Representation 
Theorem. 

\end{paragraph}

\end{chapter}


\begin{chapter}{The Salvetti Complex}\label{chap5}

Let $\A$ be an essential arrangement of hyperplanes in $\R^\ell$.  Throughout 
the remainder of this paper, we will assume that $S^{\ell-1}$ and $B^\ell$ are 
provided with the simplicial decomposition described in Chapter \ref{chap2}.  The 
following definition was introduced by W. Arvola in \cite{AR}.

\begin{definition}
We define the abstract \emph{Salvetti complex} of a hyperplane arrangement $\A$, 
denoted $\Sal(\A)$, via 
$$\Sal(\A) = \{(F,C): F \in \mathcal{L}(\A), C \in \mathcal{T}(\A), F \leq 
C\}.$$ 
\end{definition}

Recall that we use the symbol $F$ not only to denote a covector, but also the 
face of the arrangement that $F$ labels.  Similarly, we use $C$ to denote both a 
tope and the chamber that $C$ labels.  So, the pair $(F,C)$ represents a pair of 
covectors and a pair of faces.  We will only make a distiction when the context 
is unclear.

\begin{example}\label{ex5.1}
Let $\A$ be the oriented hyperplane arrangement given in Figure 
\ref{covectors2}.  The faces of the arrangement are labelled with their 
corresponding covector, given the numbering and orientation of the hyperplanes.  
\begin{figure}[h]
\begin{center}

\includegraphics[width=6cm]{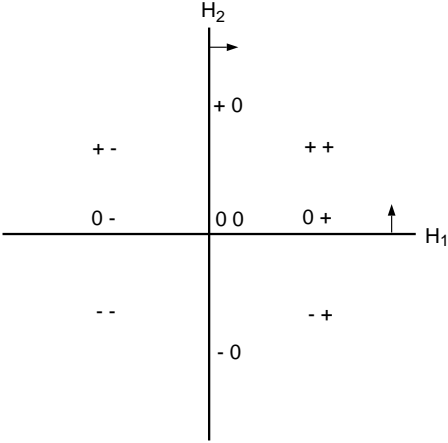}

\caption{Covectors of a hyperplane arrangement}
\label{covectors2}
\end{center}
\end{figure}
Then $\Sal(\A)$ consists of the elements\\

\bigskip

\begin{center}
$\begin{array}{llll}
(++,++) & (-+,-+) & (--,--) & (+-,+-)\\
(0+,++) & (0+,-+) & (-0,--) & (0-,+-)\\
(+0,++) & (-0,-+) & (0-,--) & (+0,+-)\\
(00,++) & (00,-+) & (00,--) & (00,+-)
\end{array}$
\end{center}
\end{example}

If $\A$ is a real hyperplane arrangement, then we define a partial order on 
$\Sal(\A)$ via
$$(F',C') \leq (F,C) \ \text{iff} \ F \leq F' \ \text{and} \ F' \circ C = C'.$$

Recall that covector composition is defined componentwise via
\begin{equation*}
(F' \circ C)_j =
	\begin{cases}
	(F')_j&  \text{if $(F')_j \neq 0$},\\
	(C)_j&  \text{if $(F')_j = 0$}.
	\end{cases}
\end{equation*}

For the remainder of this paper, we assume that $\Sal(\A)$ is provided with the 
partial ordering described above.  That is, $\Sal(\A)$ is a poset consisting of 
pairs $(F,C)$, where $F \in \mathcal{F}(\A)$, $C \in \mathcal{C}(\A)$, and $F 
\leq C$.  We may sometimes refer to a pair $(F,C)$ of $\Sal(\A)$ as a face.

\begin{example}
Let $\Sal(\A)$ consist of those elements listed in Example \ref{ex5.1}.  The 
poset for $\Sal(\A)$ is given in Figure \ref{salvetti1}.  Given this poset, we 
can realize $\Sal(\A)$ as the torus of Figure \ref{salvetti2}.
\begin{figure}[h]
\begin{center}

\includegraphics[width=9cm]{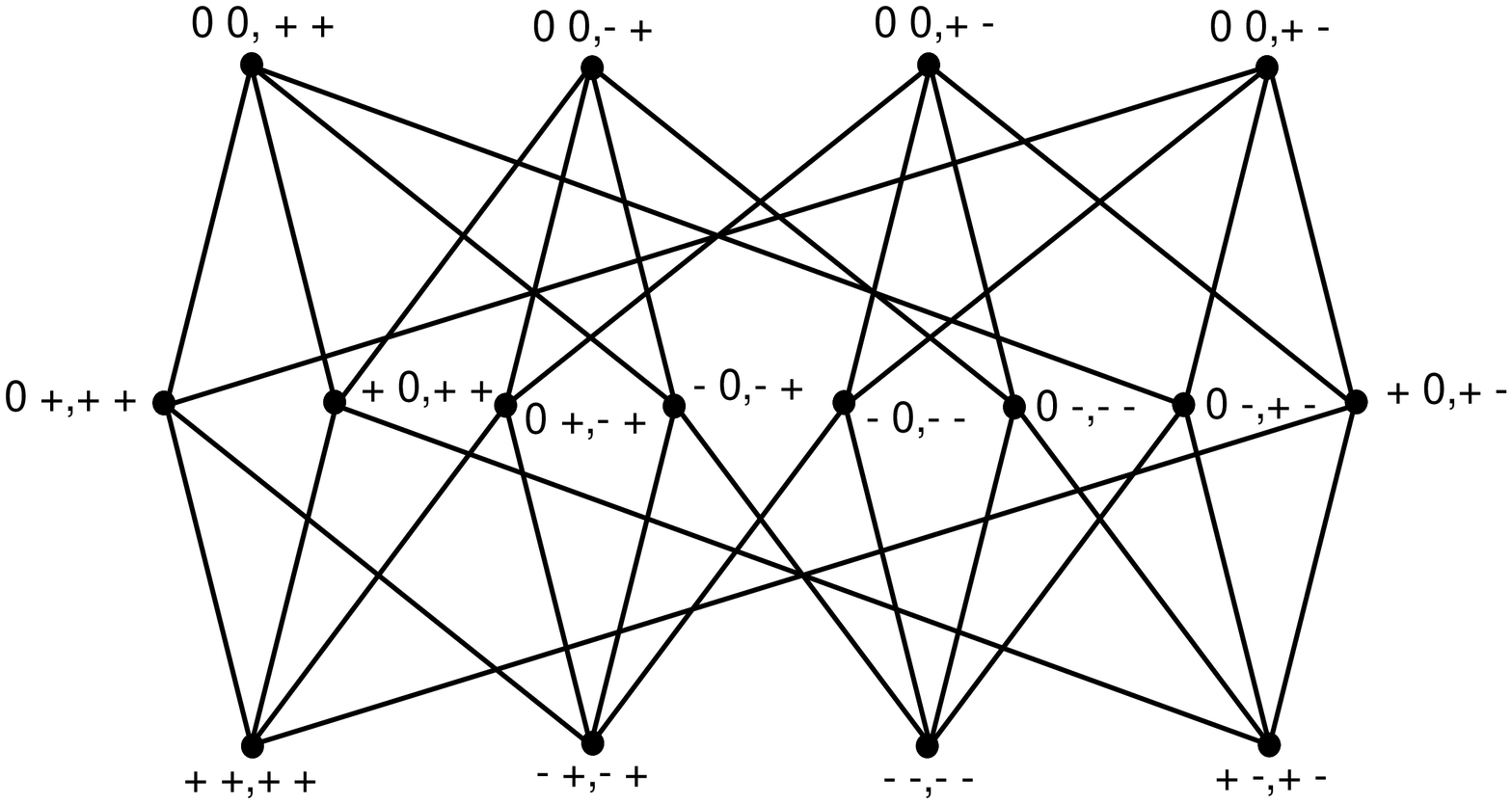}

\caption{Poset for $\Sal(\A)$}
\label{salvetti1}
\end{center}
\end{figure}

\begin{figure}[h]
\begin{center}

\includegraphics[width=9cm]{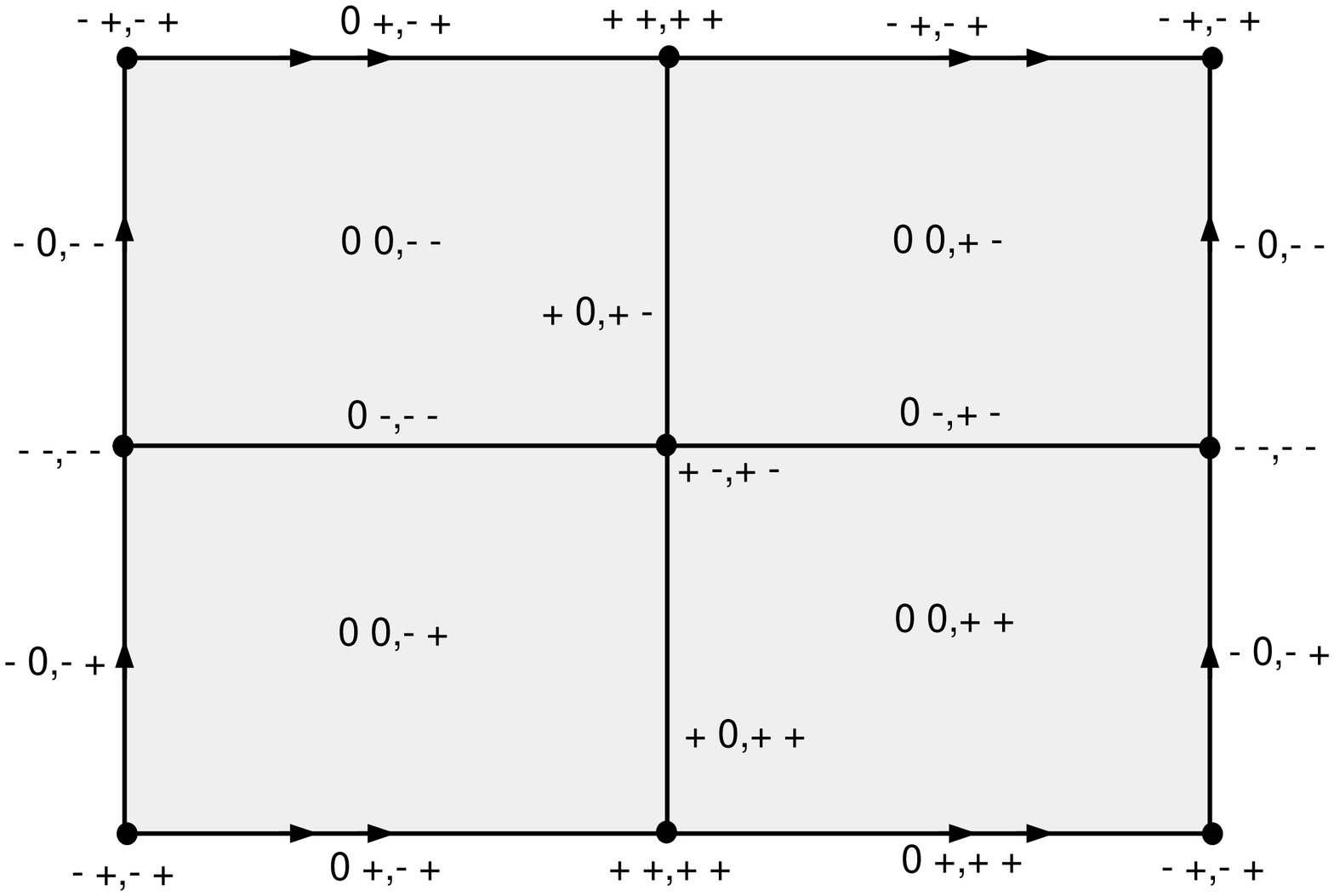}

\caption{A realization of $\Sal(\A)$}
\label{salvetti2}
\end{center}
\end{figure}
\end{example}

The face poset for $\Sal(\A)$ determines an abstract simplicial complex, by 
taking the barycentric subdivision.  A chain 
	$$\omega_0 < \cdots < \omega_n$$
of faces of $\Sal(\A)$ determines an $n$-simplex 
	$$\Omega = (\omega_0 < \cdots < \omega_n).$$ 
We will use $\Sal'(\A)$ to denote the abstract simplicial complex which results 
from this construction.  We will think of each element $\omega \in \Sal(\A)$ as 
a vertex of the simplicial complex, $\Sal'(\A)$.  Also, every $\Omega$ of 
$\Sal'(\A)$ represents both a chain of faces of $\Sal(\A)$ and a simplex.  

\bigskip

If 
	$$\Omega = (\omega_0 < \cdots < \omega_n)$$
and 
	$$\Omega' = (\omega_0' < \cdots < \omega_r')$$
are simplices of $\Sal'(\A)$,then $\Omega \leq \Omega'$ iff 
	$$\omega_0 < \cdots < \omega_n$$
is a subchain of 
	$$\omega_0' < \cdots < \omega_r'.$$   

We can realize the abstract simplicial complex $\Sal'(\A)$ in $\C^\ell$ by 
forming the \emph{concrete simplicial Salvetti complex}, denoted $|\Sal'(\A)|$.  
Recall that $\A$ determines a simplicial decomposition of $B^\ell$.  Now, for 
every pair $(F,C)$ in $\Sal(\A)$ we set 
	$$z(F,C) = x(F) + ix(C).$$
Note that since $C$ is a chamber of $\A$, $x(C) \notin H$ for each hyperplane 
$H$ of $\A$.  It follows that $z(F,C) \notin H+iH$ for any $H$.  Thus, $z(F,C) 
\in M(\A).$

\begin{lemma}\label{inequality}
Let $(F,C), (F',C') \in \Sal(\A)$ with $(F,C) < (F',C')$.  Then $F' < F$.
\end{lemma}

\begin{proof}
Assume that $F = F'$.  Then $F \circ C' = F' \circ C' = C'$ since $F' \leq C'$.  
But $F \circ C' = C$, since $(F,C) < (F',C')$.  So, $C = C'$.  This implies that 
$(F,C) = (F',C')$, which is a contradiction.  Therefore, $F' < F$.
\end{proof}

\begin{lemma}
If $\Omega = (\omega_0 < \cdots < \omega_n)$ is a simplex of $\Sal'(\A)$, then 
the points $z(\omega_0), \ldots, z(\omega_n)$ are geometrically independent.
\end{lemma}

\begin{proof}
Let $(\omega_k < \cdots < \omega_n)$ be a simplex of $\Sal'(\A)$.  Set $\omega_j 
= (F_j,C_j)$.  Then each $z(\omega_j) = x(F_j)+ix(C_j)$.  Assume that 
$x(F_{k+1}), \ldots, x(F_n)$ are geometrically independent.  Then 
$z(\omega_{k+1}), \ldots, z(\omega_n)$ are geometrically independent.  Also, 
$x(F_{k+1}), \ldots, x(F_n)$ lie in the support of $F_{k+1}$, since $F_n \leq 
\cdots \leq F_{k+1}$.  By Lemma \ref{inequality}, $F_{k+1} < F_k$.  This implies 
that $x(F_k)$ does not lie in the support of $F_{k+1}$.  So, $x(F_k), \ldots, 
x(F_n)$ are geometrically independent.  This implies that $z(\omega_k), \ldots, 
z(\omega_n)$ are geometrically independent.  This induction shows that 
$z(\omega_0), \ldots, z(\omega_n)$ are geometrically independent.   
\end{proof}

The concrete Salvetti complex $|\Sal'(\A)|$ is a simplicial complex having 
$\{z(\omega): \omega \in \Sal(\A)\}$ as its set of vertices.  The vertices 
$z(\omega_0), \ldots, z(\omega_n)$ of $|\Sal'(\A)|$ form an $n$-simplex of 
$|\Sal'(\A)|$ exactly when $\omega_0, \cdots, \omega_n$ are the elements of a 
chain $\Omega$.  If 
	$$\Omega = (\omega_0 < \cdots < \omega_n)$$
is a simplex of $\Sal'(\A)$, then we will use $z(\Omega)$ to denote 
the simplex 
	$$z(\omega_0) \vee \cdots \vee z(\omega_n)$$
of $|\Sal'(\A)|$.

\bigskip

Recall from Chapter \ref{chap2}, that we define the cone of a cell $\phi$ of $S^{\ell-1}$ 
to be
	$$K(\phi) = \{\lambda x: x \in \phi \ \text{and} \ \lambda > 0\}.$$
For ease of notation, we will denote the cone of the star of a vertex $v$ of 
$S^{\ell-1}$ by $\hat{K}(v)$.  That is, $\hat{K}(v) = K(Star(v))$.

\begin{example}
The shaded region of Figure \ref{cone1} shows the cone of the star of the vertex 
$v_0$ on $S^1$.
\begin{figure}[h]
\begin{center}

\includegraphics[width=6cm]{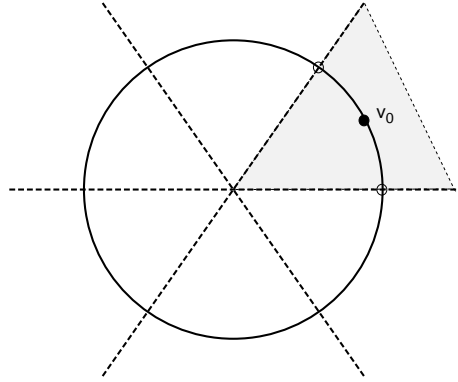}

\caption{The cone of the star of a vertex}
\label{cone1}
\end{center}
\end{figure}
\end{example}

Now, we will list a series of lemmas which will be used to prove that 
$|\Sal'(\A)|$ has the same homotopy type as the complement $M(\A)$.

\begin{lemma}\label{subset}
Let $(F,C) \in \Sal(\A)$ for some essential hyperplane arrangement $\A$.  Then
	$$F+iC_F \subseteq M(\A).$$
\end{lemma}

\begin{proof}
Let $z = x+iy \in F+iC_F$.  Then $x \in F$ and $y \in C_F$.  Since $C_F$ is the 
unique chamber of $\A_F$ which contains $C$, $y$ is not contained in any of the 
hyperplanes that $F$ may be contained in.  Hence $x+iy \notin H+iH$ for every 
hyperplane $H$ of $\A$.  So, $z \in M(\A)$.  Therefore, 
	$$F+iC_F \subseteq M(\A).$$ 
\end{proof}

\begin{lemma}\label{nugget}
Let 
	$$\phi = x(F_0) \vee \cdots \vee x(F_n)$$ 
be a simplex of $S^{\ell-1}$ with 
	$$\{0\} \neq F_0 \leq \cdots \leq F_n.$$  
Then $K(\phi) \subseteq \hat{K}(x(F_i))$ for all $i$.
\end{lemma}

\begin{proof}
Since each $x(F_i)$ is a vertex of $\phi$, $\phi \subseteq \Star(x(F_i))$ for 
each $i$.  Hence, $K(\phi) \subseteq \hat{K}(x(F_i))$ for each $i$.
\end{proof}

\begin{lemma}\label{beef}
Let $F \in \mathcal{F}(\A)$ and let $D$ be a chamber of $\A_F$.  Then there 
exists a unique chamber $C$ of $\A$ such that $C_F = D$ and $F \leq C$.
\end{lemma}

\begin{proof}
Let $F \in \mathcal{F}(\A)$ and let $D$ be a chamber of $\A_F$.  Clearly, there 
exists at least one chamber $C_0$ of $\A$ such that $(C_0)_F = D$.  Since $D$ is 
a chamber of $\A_F$, $D$ is a partial covector having one entry for each 
hyperplane $H_j$ which contains the face $F$.  That is, $D_j$ exists only when 
$F_j = 0$.  Since $(C_0)_F = D$, $(C_0)_j = D_j$ when $F_j = 0$.  Furthermore, 
if $C$ is any chamber with $C_F = D$, then $C_j = D_j$ when $F_j = 0$.  Also, if 
$F \leq C$, then $C_j = F_j$ whenever $F_j \neq 0$.  Thus, there can be at most 
one chamber $C$ satisfying $C_F = D$ and $f \leq C$ and such a chamber must be 
given by
\begin{equation*}
C_j = 
	\begin{cases}
	D_j&  \text{if $H_j \in \A_F$},\\
	F_j&  \text{if $F_j \notin \A_F$}.
	\end{cases}
\end{equation*}
On the other hand, the sign vector given by this formula is precisely $F \circ 
C_0$, which is in $\mathcal{L}(\A)$ by Theorem \ref{matroid} and (L2).  Clearly, 
$C$ is a tope.
\end{proof}

\begin{lemma}\label{wang}
Let $(F_0,C_0), \ldots, (F_n,C_n)$ be elements of $\Sal(\A)$, where 
	$$(F_0,C_0) \leq \cdots \leq (F_n,C_n).$$
Then $C_n \subseteq (C_i)_{F_i}$ for all $i$. 
\end{lemma}

\begin{proof}
Recall that $(F_i,C_i) \leq (F_{i+1}, C_{i+1})$ implies that 
	$$F_{i+1} \leq F_i \ \text{and} \ F_i \circ C_{i+1} = C_i.$$
We will prove that $C_n \subseteq (C_i)_{F_{i}}$ for all $i$ by induction.  
First, $C_n \subseteq (C_n)_{F_n}$, by definition.  Now, assume that $C_n 
\subseteq (C_{i+1})_{F_{i+1}}$.  Fix $j$ with $(F_i)_j = 0$.  Then $(F_{i+1})_j 
= 0$, since $F_{i+1} \leq F_i$.  Now, since $(F_i,C_i) \leq (F_{i+1},C_{i+1})$,
	$$F_i \circ C_{i+1} = C_i.$$
But $(F_i \circ C_{i+1})_j = (C_{i+1})_j$, by definition, since $(F_i)_j = 0$.  
So, $(C_i)_j = (C_{i+1})_j$.  Also, $(C_n)_j = (C_{i+1})_j$ since $C_n \subseteq 
(C_{i+1})_{F_{i+1}}$.  Hence, 
	$$(C_i)_j = (C_{i+1})_j = (C_n)_j\text{,}$$
whenever $H_j$ is a hyperplane of $\A_{F_i}$.  Therefore, $C_n \subseteq 
(C_i)_{F_i}$.  
\end{proof}

Now, we are ready to state the main result of this paper, the proof of which 
will rely on the above lemmas and the Nerve Theorem.  The proof that we give 
follows the same framework as a proof given by L. Paris in \cite{PA}.  In 
\cite{PA}, Paris constructs a concrete Salvetti complex from pairs $(F,C)$, 
where $F$ is a face of an essential arrangement $\A$ and $C$ is an arbitrary 
chamber.  He does not require that $F$ be a face of $C$, as we have.  From each 
pair $(F,C)$, Paris forms the vertex
	$$z(F,C) = x(F) + ix(C_F)$$
in $\C^\ell$.  He points out that $z(F_1,C_1) = z(F_2,C_2)$ iff $F_1 = F_2$ and 
$(C_1)_{F_1} = (C_2)_{F_2}$.  Our Lemma \ref{beef} shows that we get the same 
vertex set as Paris.

\begin{theorem}[{\rm Salvetti} \cite{SA2}, {\rm Paris} \cite{PA}]\label{hmtpy}
Let $\A$ be an oriented hyperplane arrangement.  Then $|\Sal'(\A)|$ has the same 
homotopy type as $M(\A)$.
\end{theorem}

\begin{proof}
First, assume that $B^\ell$ is provided with the simplicial decompostion 
determined by $\A$.  Now, for every $\omega \in \Sal(\A)$, we will associate an 
open convex subset $U(\omega)$ of $M(\A)$.  Let $(F,C)$ be an element of 
$\Sal(\A)$.  If $F = \{0\}$, then we set 
	$$U(F,C) = \R^\ell +iC.$$
If $F \neq \{0\}$, then we set
	$$U(F,C) = \hat{K}(x(F)) + iC_F.$$

By Lemma \ref{beef}, each the $U(\omega)$ are distinct.  Also, $U(\omega)$ is an 
open convex set for each $\omega \in \Sal(\A)$.  Then any nonempty intersection 
of elements of 
	$$\mathcal{U} = \{U(\omega): \omega \in \Sal(\A)\}$$
will be convex and thus contractible.  It is also clear that $z(\omega) \in 
U(\omega)$ for every $\omega$ of $\Sal(\A)$. 

\bigskip

Now, we will prove the following four assertions successively.
\begin{enumerate}[label=\rm{(\arabic*)}]
\item $U(\omega) \subseteq M(\A)$ for every $\omega \in \Sal(\A)$. 
\item $M(\A) \subseteq \bigcup_{\omega \in \Sal(\A)} U(\omega)$. 
\item Let $z(\omega_0), \ldots, z(\omega_n)$ be $(n+1)$ distinct vertices 
of $|\Sal'(\A)|$.  If 
	$$U(\omega_0) \cap \cdots \cap U(\omega_n) \neq \emptyset\text{,}$$
then $z(\omega_0), \ldots, z(\omega_n)$ are the vertices of a simplex 
$z(\Omega)$ of $|\Sal'(\A)|$. 
\item If $z(\omega_0), \ldots, z(\omega_n)$ are the vertices of a simplex 
$z(\Omega)$ of $|\Sal'(\A)|$, then 
	$$U(\omega_0) \cap \cdots \cap U(\omega_n) \neq \emptyset.$$ 
\end{enumerate}

Clearly, (1) and (2) prove that $\mathcal{U}$ is an open covering of $M(\A)$.  
Also, assertions (3) and (4) prove that $\Sal'(\A)$ is the nerve of 
$\mathcal{U}$.  This implies, by the Nerve Theorem (Theorem \ref{nerve}), that 
$|\Sal'(\A)|$ has the same homotopy type as $M(\A)$.

\bigskip
 
(1)  Let $(F,C) \in \Sal(\A)$.  If $F = \{0\}$, then 
	$$U(F,C) = \R^\ell+iC.$$  
Since $z = x+iy \in H+iH$ iff $x \in H$ and $y \in H$, $U(F,C)$ is clearly 
contained in $M(\A)$.  Now, assume that $F \neq \{0\}$.  Then 
	$$U(F,C) = \hat{K}(x(F)) + iC_F.$$
Pick $z=x+iy \in U(F,C)$.  Then $x \in \hat{K}(x(F))$.  Thus, there exists a 
simplex $\phi$ of $S^{\ell-1}$ such that $x(F)$ is a vertex of $\phi$ and $x \in 
K(\phi)$.  Say 
	$$\phi = x(F_0) \vee \cdots \vee x(F_n)\text{,}$$
with 
	$$\{0\} \neq F_0 < \cdots < F_n\text{,}$$ 
where $F = F_i$ for some $i$.  By Lemma \ref{word}, $K(\phi) \subseteq F_n$, 
which implies that $x \in F_n$.  Also, $F \leq F_n$.  So, $C_F \subseteq 
C_{F_n}$, which implies that $y \in C_{F_n}$.  Therefore, 
	$$z = x+iy \in F_n+iC_{F_n}.$$
By Lemma \ref{subset}, $F_n+iC_{F_n} \subseteq M(\A)$ and so $z = x+iy \in 
M(\A)$.  Hence, $U(\omega) \subseteq M(\A)$ for every $\omega \in \Sal(\A)$.

\bigskip

(2)  Let $z=x+iy \in M(\A)$.  If $x=0$, then $x \in H$ for every $H \in \A$.  
Thus, $y \notin H$ for every $H \in \A$, since $x+iy \in M(\A)$.  So, there 
must exist a chamber $C$ of $\A$ such that $y \in C$.  Therefore, 
	$$z=x+iy \in \R^\ell+iC.$$
But $(\R^\ell+iC) = U(\{0\},C)$.  Hence, 
	$$z=x+iy \in U(\{0\},C).$$
Now, assume that $x \neq 0$.  Then there exists a simplex $\phi$ 
of $S^{\ell-1}$ such that $x \in K(\phi)$.  Say 
	$$\phi = x(F_0) \vee \cdots \vee x(F_n)$$
where 
	$$\{0\} \neq F_0< \cdots <F_n.$$  
Since $x \in K(\phi) \subseteq F_n$ and $z=x+iy \in M(\A)$, there is no 
hyperplane $H \in \A$ containing $F_n$ which also contains $y$.  Hence, there 
exists a chamber $D$ of $\A_{F_n}$ such that $y \in D$.  By Lemma \ref{beef}, 
there exists a chamber $C$ of $\A$ such that $C_{F_n}=D$ and $F_n \leq C$, so 
that $(F_n,C) \in \Sal(\A)$.  Then 
	$$z=x+iy \in K(\phi)+iC_{F_n}.$$  
Also, by Lemma \ref{nugget}, $K(\phi) \subseteq \hat{K}(x(F_n))$.  So, 			
	$$K(\phi)+iC_{F_n} \subseteq \hat{K}(x(F_n)) + iC_{F_n} = U(F_n,C).$$  
Therefore, 
	$$z = x+iy \in U(F_n,C).$$
Thus, $M(\A) \subseteq \bigcup_{\omega \in \Sal(\A)}U(\omega)$.

\bigskip

(3)  Let $z(\omega_0), \ldots, z(\omega_n)$ be $(n+1)$ distinct vertices of 
$|\Sal'(\A)|$ such that 
	$$U(\omega_0) \cap \cdots \cap U(\omega_n) \neq \emptyset.$$
Then $\omega_0, \ldots, \omega_n$ are $(n+1)$ distinct elements of $\Sal(\A)$.  
Set 
	$$\omega_i = (F_i,C_i)$$
for all $i$ and let 
	$$z = x+iy \in \bigcap_{i=0}^nU(F_i,C_i).$$
We have two cases to consider:  $F_i = \{0\}$ for some $i$ and $F_i \neq \{0\}$ 
for all $i$.

\bigskip

\emph{Case A}:  Without loss of generality, assume that $F_0 = \{0\}$.  First, 
suppose that there exists an $j$ such that $F_j = \{0\}$.  Then
	$$z = x+iy \in U(\omega_0) \cap U(\omega_j) = (\R^\ell+iC_0) \cap 		
	(\R^\ell+iC_j).$$
But then $y \in C_0 \cap C_j$.  So, $C_0 \cap C_j \neq \emptyset$, which implies 
that $C_0 = C_j$.  Then 
	$$\omega_0 = (\{0\},C_0) = (\{0\},C_j) = \omega_j.$$
But this contradicts the fact that $\omega_0, \ldots, \omega_n$ are $(n+1)$ 
distinct elements of $\Sal(\A)$.  Therefore, $F_j \neq \{0\}$ for all $j > 0$.  
Now, since $x \in \hat{K}(x(F_i))$ for each $i$, there exists a simplex $\phi_i$ 
of $S^{\ell-1}$ such that $x(F_i)$ is a vertex of $\phi_i$ and $x \in 
K(\phi_i)$, for all $i$.  This implies that 
	$$\phi_1 = \cdots = \phi_n\text{,}$$
since 
	$$\{K(\phi): \phi \ \text{a simplex of} \ S^{\ell-1}\}$$
is a partition of $\R^\ell-\{0\}$.  Hence, $x(F_1), \ldots, x(F_n)$
are the vertices of some simplex $\phi$ of $S^{\ell-1}$.  

\bigskip

Without loss of generality, assume that 
	$$\{0\} = F_0 < F_1 \leq \cdots \leq F_n.$$
Now, we need to show that $F_{i+1} \circ C_i = C_{i+1}$, for all $i$. We will 
show that 
	$$(F_{i+1} \circ C_i)_j = (C_{i+1})_j$$
for each $j$.  Recall that
\begin{equation*}
(F_{i+1} \circ C_i)_j =
	\begin{cases}
	(F_{i+1})_j&  \text{if $(F_{i+1})_j \neq 0$},\\
	(C_i)_j&  \text{if $(F_{i+1})_j = 0$}.
	\end{cases}
\end{equation*}
Fix $j$.  First, assume that $(F_{i+1})_j \neq \{0\}$.  Then 
	$$(F_{i+1} \circ C_i)_j = (F_{i+1})_j.$$
But since $F_{i+1} \leq C_{i+1}$, $(F_{i+1})_j = (C_{i+1})_j$.  Hence, 
	$$(F_{i+1} \circ C_i)_j = (C_{i+1})_j.$$  
Now, assume that $(F_{i+1})_j = \{0\}$.  Then $F_{i+1}$ lies inside the 
hyperplane $H_j$.  Also, 
	$$(F_{i+1} \circ C_i)_j = (C_i)_j$$
by definition.  We need to show that $(C_i)_j = (C_{i+1})_j$.  We know that 		
	$$(C_i)_{F_i} \cap (C_{i+1})_{F_{i+1}} \neq \emptyset$$
since both factors contain $y$, and that $F_i \leq F_{i+1}$.  Then $\A_{F_{i+1}} 
\subseteq \A_{F_i}$, which implies that $(C_i)_{F_i} \subseteq (C_i)_{F_{i+1}}$. 
Since 
	$$(C_i)_{F_i} \cap (C_{i+1})_{F_{i+1}} \neq \emptyset$$
and $(C_i)_{F_i} \subseteq (C_i)_{F_{i+1}}$, 
	$$(C_i)_{F_{i+1}} \cap (C_{i+1})_{F_{i+1}} \neq \emptyset.$$  
So, 
	$$(C_i)_{F_{i+1}} = (C_{i+1})_{F_{i+1}}.$$ 
Since $H_j \in \A_{F_{i+1}}$, $C_i$ and $C_{i+1}$ lie on the same side of $H_j$. 
Hence, $(C_i)_j = (C_{i+1})_j$.  Therefore, 
	$$(F_{i+1} \circ C_i)_j = (C_{i+1})_j.$$  
Hence 
	$$F_{i+1} \circ C_i = C_{i+1}.$$  
So, $(F_{i+1}, C_{i+1}) \leq (F_i,C_i)$ for all $i$.  That 
is, 
	$$(F_n,C_n) \leq \cdots \leq (F_1,C_1) \leq (F_0,C_0).$$

\bigskip

\emph{Case B}:  Assume that $F_0 \neq \{0\}$.  Since $x \in \hat{K}(x(F_i))$ for 
each $i$, there exists a simplex $\phi_i$ of $S^{\ell-1}$ such that $x(F_i)$ is 
a vertex of $\phi_i$ and $x \in K(\phi_i)$, for $i$.  This implies that 
	$$\phi_0 = \phi_1 = \cdots = \phi_n\text{,}$$
and so, $x(F_0), \ldots, x(F_n)$ are the vertices of some simplex $\phi$ of 
$S^{\ell-1}$, just as in the previous case.  

\bigskip

Without loss of generality, assume that 
	$$\{0\} \neq F_0 < F_1 \leq \cdots \leq F_n.$$
The proof that $F_{i+1} \circ C_i = C_{i+1}$, for each $i$ is identical to the 
proof above.  So, we have that
	$$(F_n,C_n) \leq \cdots \leq (F_1,C_1) \leq (F_0,C_0).$$

\bigskip

Therefore, in either case, we have that 
	$$(F_n,C_n),\ldots, (F_1,C_1), (F_0,C_0)$$ 
are the vertices of an $(n+1)$ simplex $\Omega$ of $\Sal'(\A)$.  It follows that 
	$$z(F_n,C_n), \ldots, z(F_1,C_1), z(F_0,C_0)$$
are the vertices of an $(n+1)$ simplex of $|\Sal'(\A)|$. 

\bigskip

(4)  Let 
	$$z(\Omega) = z(\omega_0) \vee \cdots \vee z(\omega_n)$$ 
be a simplex of $|\Sal'(\A)|$, where 
	$$\omega_0 \leq \cdots \leq \omega_n.$$  
Then 
	$$\Omega = (\omega_0 < \cdots < \omega_n)$$
is a simplex of $\Sal'(\A)$.  Set	
	$$\omega_i = (F_i,C_i)$$
for each $i$.  Then 
	$$F_0 \geq \cdots \geq F_n.$$  
Note the change in the inequality above.  Let 
	$$\phi = x(F_0) \vee x(F_1) \vee \cdots \vee x(F_n).$$
Then $\phi$ is a simplex of $B^\ell$.  We have two cases to consider:  $F_n = 
\{0\}$ and $F_n \neq \{0\}$.

\bigskip

\emph{Case A}:  Assume that $F_n = \{0\}$.  Consider the simplex 
	$$\phi' = x(F_0) \vee \cdots \vee x(F_{n-1})$$
of $S^{\ell-1}$.  Let $x \in K(\phi')$ and $y \in C_n$.  Set 
	$$z = x+iy.$$
Clearly, we have 
	$$z \in \R^\ell+iC_n = U(F_n,C_n).$$
Since $x \in K(\phi')$, $x \in \hat{K}(x(F_i))$ for each $i = 0, 1, \ldots, 
n-1$, by Lemma \ref{nugget}.  Also, $y \in (C_i)_{F_i}$ for $i = 0, 1, 
\ldots, n$, by Lemma \ref{wang}.  Therefore, 
	$$z = x+iy \in \hat{K}(x(F_i))+i(C_i)_{F_i} = U(F_i,C_i)$$ 
for $i = 0, \ldots, n$.  So, it follows that 
	$$U(\omega_0) \cap \cdots \cap U(\omega_n) \neq \emptyset.$$

\bigskip

\emph{Case B}:  Assume that $F_n \neq \{0\}$.  Then $\phi$ is a simplex of 
$S^{\ell-1}$.  Let $x \in K(\phi)$ and $y \in C_n$.  Set 
	$$z = x+iy.$$
Since $x \in K(\phi)$, $x \in \hat{K}(F_i)$ for every $i$, by Lemma 
\ref{nugget}.  Also, $y \in (C_i)_{F_i}$ for each $i$, by Lemma \ref{wang}.  
Therefore, 
	$$z = x+iy \in \hat{K}(x(F_i))+i(C_i)_{F_i} = U(F_i,C_i)$$
for each $i$.  It follows that
	$$U(\omega_0) \cap \cdots \cap U(\omega_n) \neq \emptyset.$$
\end{proof}

\begin{corollary}
Let $\A$ be an essential hyperplane arrangement.  Then 
	$$|\Sal'(\A)| \subseteq M(\A).$$
\end{corollary}

\begin{proof}
The result follows from the fact that each $U(\omega_i) \in \mathcal{U}$ is 
convex and from steps (1) and (4) of the proof of Theorem \ref{hmtpy}.
\end{proof}

Let $\A$ be an essential hyperplane arrangement.  Given $|\Sal'(\A)|$, we can 
easily form the realization $|\Sal(\A)|$ of $\Sal(\A)$.  For each $\omega \in 
\Sal(\A)$ form the cell 
	$$|\omega| = \bigcup_{z(\omega) \leq z(\Omega)}z(\Omega)\text{,}$$
where $\Omega$ is a simplex of $\Sal'(\A)$.  This construction yeilds a regular 
cell complex which is a sort of anti-order complex.  Also, it is clear that 
$|\Sal'(\A)|$ and $|\Sal(\A)|$ are homeomorphic as topological spaces.  So, by 
Proposition \ref{homeo}, we have the following corollary.     

\begin{corollary}
The regular cell complex $|\Sal(\A)|$ has the same homotopy type as $M(\A)$.  \hfill $\Box$
\end{corollary}

Recall that in Chapter \ref{chap3} we proved that the complement of the braid arrangement, 
$M(\A_{\ell-1})$, is isomorphic to the labelled configuration space 
$\hat{F_\ell}(\C)$.  This implies the following result.

\begin{corollary}
$|\Sal(\A_{\ell-1})|$ has the same homotopy type as $\hat{F_\ell}(\C)$. \hfill $\Box$
\end{corollary}   

\begin{paragraph}{Complex Covectors}
There is an alternative description of $\Sal(\A)$ which we will now discuss.  
If $\A$ consists of the hyperplanes $H_1, \ldots, H_n$, then we can form the set 
of \emph{complex covectors} for the complexified arrangement $\A_\C$, where each 
complex covector is an element of $\{0,+,-,i,-i\}^n$.  Let $\alpha_j$ be the 
defining equation for the complex hyperplane $H_j+iH_j$.  For each $z \in 
\C^\ell$ we associate a vector $X_z$, which we define componentwise via
\begin{equation*}
(X_z)_j =
	\begin{cases}
	0&  \text{if $\alpha_j(z) = 0$},\\
	+&  \text{if $\Real(\alpha_j(z)) > 0$},\\
	-&  \text{if $\Real(\alpha_j(z)) < 0$},\\
	i&  \text{if $\Real(\alpha_j(z)) = 0$ and $\Imag(\alpha_j(z)) > 0$},\\
	-i& \text{if $\Real(\alpha_j(z)) = 0$ and $\Imag(\alpha_j(z)) < 0$}.
	\end{cases}
\end{equation*}  
That is, the defining equation for $H_j+iH_j$ determines the $j^{th}$ entry for 
each vector  $X_z$.  The complex covectors for $\A_\C$ consist of all the 
vectors $X_z$ where $z \in \C^\ell$.  Certainly, the set of complex covectors is 
finite.  Note that the complex covectors which label the complement $M(\A)$ 
consist of only nonzero entries.

\begin{figure}[h]
\begin{center}

\includegraphics[width=4cm]{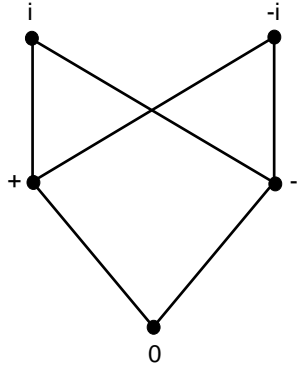}

\caption{The partial ordering on $\{0,+,-,i,-i\}$}
\label{complexpo}
\end{center}
\end{figure}

Now, define the partial order ``$\leq$" on the set $\{0,+,-,i,-i\}$ via $0<+$, 
$0<-$, $+ < i$, $- < -i$, $- < i$, and $+ < -i$ with $+$ and $-$ incomparable 
and $i$ and $-i$ incomparable.  See Figure \ref{complexpo}.

\bigskip

This induces a partial order on complex covectors in which complex covectors are 
compared componentwise.  That is, if $X$ and $X'$ are complex covectors of a 
complexified arrangement $\A_\C$, then 		
	$$X \leq X' \ \text{iff} \ X_j \leq X_j'$$
for all $j$.

\bigskip

We can convert each pair $(F,C) \in \Sal(\A)$ to a complex covector $F^C$.  Let 
$\mathcal{L}(\A_\C)$ denote the set of nowhere zero complex covectors for a 
complexified arrangement $\A_\C$.  Now, define $f: \Sal(\A) \to 
\mathcal{L}(\A_\C)$ via 
	$$f(F,C) = F^C \text{,}$$
where
\begin{equation*}
(F^C)_j =
	\begin{cases}
	F_j&  \text{if $F_j \neq 0$},\\
	i&  \text{if $F_j = 0$ and $C_j = +$},\\
	-i& \text{if $F_j = 0$ and $C_j = -$}.
	\end{cases}
\end{equation*}

\begin{example}
Let $\A$ be the oriented hyperplane arrangment of Figure \ref{covectors2}.  Then 
the correspondence between $(F,C)$ and $F^C$ for $\A$ is given below.

\bigskip

\begin{center}
$\begin{array}{lll}
++,++ & \longleftrightarrow & +,+\\
0+,++ & \longleftrightarrow & i,+\\
+0,++ & \longleftrightarrow & +,i\\
00,++ & \longleftrightarrow & i,i\\
-+,-+ & \longleftrightarrow & -,+\\
0+,-+ & \longleftrightarrow & -i,+\\
-0,-+ & \longleftrightarrow & -,i\\
00,-+ & \longleftrightarrow & -i,i\\
--,-- & \longleftrightarrow & -,-\\
-0,-- & \longleftrightarrow & -,-i\\
0-,-- & \longleftrightarrow & -i,-\\
00,-- & \longleftrightarrow & -i,-i\\
+-,+- & \longleftrightarrow & +,-\\
0-,+- & \longleftrightarrow & i,-\\
+0,+- & \longleftrightarrow & +,-i\\
00,+- & \longleftrightarrow & i,-i
\end{array}$
\end{center}
\end{example}

\begin{lemma}
The function $f$ is injective.
\end{lemma}

\begin{proof}
Let $f$ be defined as above.  Assume that ${F_1}^{C_1} = {F_2}^{C_2}$.  Then 
$({F_1}^{C_1})_j = ({F_2}^{C_2})_j$ for each $j$.  We have three 
cases to consider.

\bigskip

\emph{Case A}:  Assume that $({F_1}^{C_1})_j = (F_1)_j$.  Then $({F_2}^{C_2})_j 
= (F_1)_j$. This forces $({F_2}^{C_2})_j = (F_2)_j$, since $({F_2}^{C_2})_j$ is 
real.  Hence $(F_1)_j = (F_2)_j$.  Also,  since $(F_1,C_1), (F_2,C_2) \in 
\Sal(\A)$, $(F_1)_j \leq (C_1)_j$ and $(F_2)_j \leq (C_2)_j$.  Thus, $(F_1)_j = 
(C_1)_j$ and $(F_1)_j = (C_2)_j$, since $(F_1)_j = (F_2)_j \neq 0$.  Hence 
$(C_1)_j = (C_2)_j$.

\bigskip

\emph{Case B}:  Assume that $({F_1}^{C_1})_j = i$.  Then $({F_2}^{C_2})_j = i$.  
Since $({F_1}^{C_1})_j = i$, $(F_1)_j = 0$ and $(C_1)_j = +$.  Also, since 
$({F_2}^{C_2})_j = i$, $(F_2)_j = 0$ and $(C_2)_j = +$.  Thus, $(F_1)_j = 
(F_2)_j$ and $(C_1)_j = (C_2)_j$.

\bigskip

\emph{Case C}:  Assume that $({F_1}^{C_1})_j = -i$.  Then $(F_1)_j = (F_2)_j$ 
and $(C_1)_j = (C_2)_j$, just as in Case B.

\bigskip

So, in any case, $(F_1)_j = (F_2)_j$ and $(C_1)_j = (C_2)_j$.  This implies that 
$(F_1,C_1) = (F_2,C_2)$.  Therefore, $f$ is injective.
\end{proof}

\begin{lemma}
The function $f$ is surjective.
\end{lemma}

\begin{proof}
Let $X_z \in \mathcal{L}(\A_\C)$.  Define $(F,C)$ via
\begin{equation*}
(F_j,C_j) =
	\begin{cases}
	(+,+)&  \text{if $(X_z)_j = +$},\\
	(-,-)&  \text{if $(X_z)_j = -$},\\
	(0,+)&  \text{if $(X_z)_j = i$},\\
	(0,-)&  \text{if $(X_z)_j = -i$}.
	\end{cases}
\end{equation*}
Then clearly, $f(F,C) = X_z$.  We must show that the covector $F$ labels a face 
of $\A$, the covector $C$ labels a chamber of $\A$, and that $F \leq C$, so that 
$(F,C)$ will be an element of $\Sal(\A)$.  Set $z = x+iy$.  From the definition 
of $(F,C)$, one sees that $F$ labels the face of $\A$ containing $x$, and $C$ 
labels the unique chamber having $F$ as a face, such that $C_F$ contains $y$, 
which exists by Lemma \ref{beef}.  This implies that $(F,C) \in \Sal(\A)$.  
Therefore, $f$ is surjective.  
\end{proof}

The previous two lemmas show that there is a bijection between $\Sal(\A)$ and 
$\mathcal{L}(\A_C)$.

\begin{figure}[h]
\begin{center}

\mbox{\subfigure[]{\includegraphics[width=3cm]{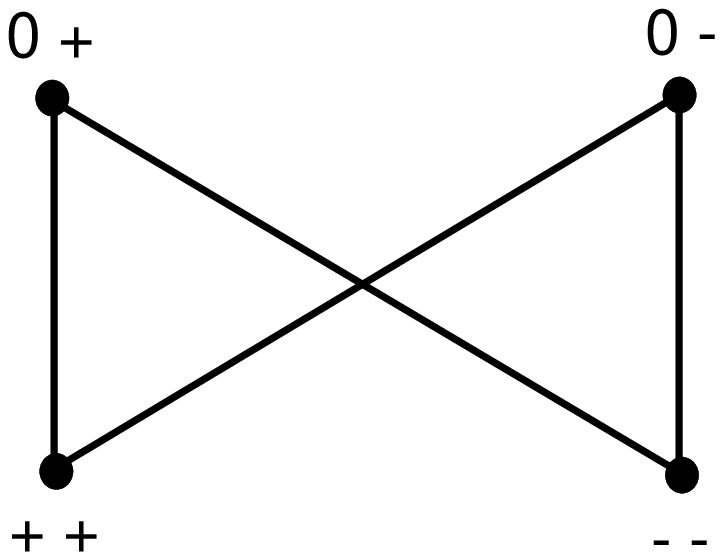}}\quad
\subfigure[]{\includegraphics[width=3cm]{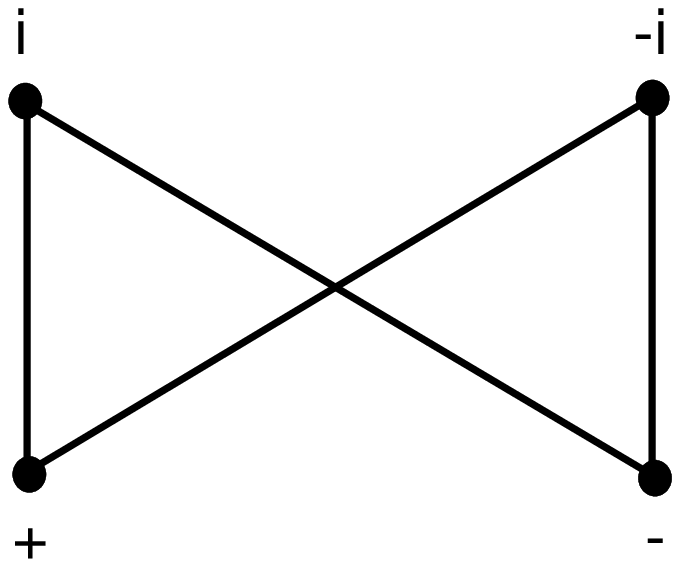}}}

\caption{Partial ordering on components}
\label{compwisepo}
\end{center}
\end{figure}

\begin{proposition}
$\Sal(\A)$ and $\mathcal{L}(\A_\C)$ are isomorphic as posets.
\end{proposition}

\begin{proof}
Since there is a bijection between $\Sal(\A)$ and $\mathcal{L}(\A_\C)$, all that 
we need to show is that the function $f$ preserves order.  Let $(F,C) \in 
\Sal(\A)$.  Then since $F \leq C$, $(F_j,C_j)$ is one of $(0,+), (0,-), (+,+), 
(-,-)$.  Furthermore, the ordering $(F,C) \leq (F',C')$ defined by $F' \leq F$ 
and $F \circ C' = C$ can be obtained from a coordinate-wise ordering according 
to Figure \ref{compwisepo}(a).  Since $f$ is defined coordinate-wise via

\begin{center}
$\begin{array}{lll}
(0,+) & \longrightarrow & i\\
(0,-) & \longrightarrow & -i\\
(+,+) & \longrightarrow & +\\
(-,-) & \longrightarrow & -
\end{array}$
\end{center}

\noindent{it is obvious when we compare Figures \ref{compwisepo}(a) and 
\ref{compwisepo}(b) that $f$ preserves order.  Therefore, $\Sal(\A)$ and 
$\mathcal{L}(\A_\C)$ are isomorphic as posets.}
\end{proof}

The above theorem implies that we may use the set of nowhere zero complex 
covectors $\mathcal{L}(\A_\C)$ as an alternate description of $\Sal(\A)$. 

\end{paragraph}

\end{chapter}


\begin{chapter}{Orbit Complexes}\label{chap6}

Let $\A$ be a reflection arrangement in $\R^\ell$ with reflection group $W$.  
Recall that $S_\A'$ is the simplicial decomposition of $S^{\ell-1}$ determined 
by $\A$.  Also recall that in Chapter \ref{chap3}, we proved that $S_\A'$ can be 
constructed to be $W$-invariant.

\bigskip

The action of $W$ on $\A$ determines a group action on $\A_\C$, the complexified 
version of $\A$.  To be more precise, let $z=x+iy \in \C^\ell$, then each $s \in 
W$ acts on $z$ via
	$$s(z) = s(x+iy) = s(x)+is(y).$$
Now, let $(F,C) \in \Sal(\A)$.  Then each $s \in W$ acts on $z(F,C)$ via
	$$s(z(F,C)) = s(x(F))+is(x(C)).$$
Since $S_\A'$ is $W$-invariant and $W$ acts as above, the vertex set of 
$|\Sal'(\A)|$ is $W$-invariant.
\bigskip

Since $W$ preserves the order of the faces of $\Sal(\A)$ and each $s \in W$ is 
linear, it is clear that each $s \in W$ will map a simplex to a simplex.  
This implies that $|\Sal'(\A)|$ is also $W$-invariant.

\bigskip

Using Theorem \ref{hmtpy}, one can construct a homotopy 
	$$H: M(\A) \times [0,1] \to M(\A)$$
between the identity map on $M(\A)$ and a retraction onto $|\Sal'(\A)|$.  In 
\cite{SA2}, M. Salvetti constructed such an $H$ using a collection of 
straight-line homotopies.  So, we can choose $H$ to be $W$-\emph{equivariant} on 
$M(\A)$, since $W$ acts linearly.  That is, we can choose $H$ such that
	$$H(s(z)) = s(H(z))$$
for any $s \in W$ and $z \in M(\A)$.  This implies that $H$ induces a 
well-defined homotopy 
	$$\hat{H}: M(\A)/W \times [0,1] \to M(\A)/W$$
from the identity map on $M(\A)/W$ to a retraction onto $|\Sal'(\A)|/W$.  Thus, 
we have the following result.  

\begin{theorem}
$M(\A)/W$ and $|\Sal'(\A)|/W$ are homotopy equivalent. \hfill $\Box$
\end{theorem}

The next corollary follows from Proposition \ref{homeo}.

\begin{corollary}\label{cor}
$M(\A)/W$ and $|\Sal(\A)|/W$ are homotopy equivalent.  \hfill $\Box$
\end{corollary}

Now, we will discuss how this relates to the braid arrangement.  Recall that the 
finite reflection group corresponding to $\A_{\ell-1}$ is $S_\ell$.  
Also, recall that $M(\A_{\ell-1})$ is homotopy equivalent to $\hat{F}_\ell(C)$.  
Clearly, $S_\ell$ acts on $\hat{F}_\ell(C)$ by permuting the labels in each 
point configuration.  That is, all labelled configurations of the same $\ell$ 
points are in the same orbit under the action of $S_\ell$.  In other words, 
$S_\ell$ ``removes" the labels from the point configurations.  So, all of the 
labelled configurations from Figure \ref{labelled} are in the same orbit under 
the action of $S_3$. 

\begin{definition}
Define $F_\ell(\C)$ to be the space of unlabelled configurations of $\ell$ 
distinct points in $\C$.
\end{definition}

\begin{example}
Figure \ref{unlabelled} shows two examples of unlabelled configurations of 
$F_3(\C)$.
\begin{figure}[h]
\begin{center}

\includegraphics[width=4cm]{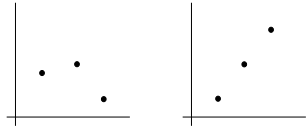}

\caption{Unlabelled point configurations of $F_3(\C)$}
\label{unlabelled}
\end{center}
\end{figure}
\end{example}

\begin{proposition}
$\hat{F}_\ell(\C)/S_\ell$ is homeomorphic to $F_\ell(\C)$.  \hfill $\Box$
\end{proposition}

The following corollary follows immediately.

\begin{corollary}
$M(\A_{\ell-1})/S_\ell$ is homeomorphic to $F_\ell(\C)$. \hfill $\Box$
\end{corollary}

The above corollary in conjuction with Corollary \ref{cor} imply the following.

\begin{corollary}
$F_\ell(\C)$ and $|\Sal(\A_{\ell-1})|/S_\ell$ are homotopy equivalent. \hfill $\Box$
\end{corollary}

That is, the Salvetti complex of the braid arrangement in $\R^\ell$ modulo the 
action of the symmetric group is a cell complex which is homotopy equivalent to 
the space of unlabelled configurations of $\ell$ distinct points.

\bigskip

Now, we wish to describe the construction of $|\Sal(\A)|/W$ for all arrangements 
in $\R^2$ which have finite reflection groups.  The construction that we give 
can be generalized for higher dimensions.  For a complete proof of the general 
construction, see \cite{SA1}.

\bigskip

Let $\A$ be an essential reflection arrangement in $\R^2$ with reflection group 
$W$.  Then we form the \emph{dual complex} to $\A$, denoted $D(\A)$, in the 
following way.  Fix a base chamber $C_0$.  Pick a point $x(C_0) \in C_0$.  For 
every other chamber $C$ of $\A$ pick $x(C) \in C$ so that the resulting set of 
points is invariant under the action of $W$.  We can do this in the same way 
that we constructed $S_A$ to be $W$-invariant.  These points will be the 
vertices of $D(\A)$.  Now, join a pair of vertices $x(C_i)$ and $x(C_j)$ by an 
edge exactly when $C_i$ and $C_j$ share a common wall.  This is the edge dual to 
the face of $\A$ which it crosses.  Finally, form the $2$-cell bounded by the 
polygon of dual edges.  This is the $2$-cell dual to $\{0\}$.  It should be 
clear that if we construct $D(\A)$ in this way, then $D(\A)$ is $W$-invariant.  

\begin{example}
Figure \ref{dual} depicts the dual complex of $\A_2$.  Recall that this 
picture of $\A_2$ is a projection onto the plane defined by
	$$x_1+x_2+x_3=0$$
perpendicular to the line 
	$$\{x_1=x_2=x_3\}.$$
\begin{figure}[h]
\begin{center}

\includegraphics[width=6cm]{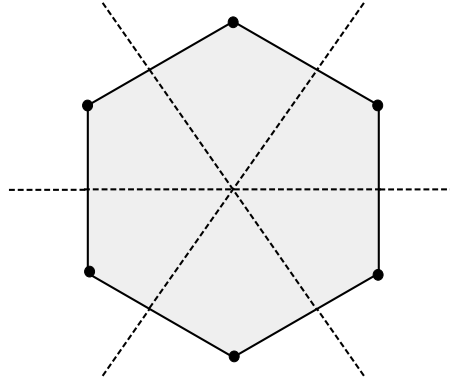}

\caption{Dual complex for $\A_2$}
\label{dual}
\end{center}
\end{figure}
\end{example}

Now, we form the cell complex $|\Sal(\A)|/W$ from $D(\A)$ in the following way:
\begin{enumerate}[label=\rm{(\roman*)}]
\item  Identify all vertices of $D(\A)$.
\item  Identify the edges dual to the faces $F$ and $F'$ exactly when $F$ 
and $F'$ are in the same orbit under the action of $W$.
\end{enumerate}
The resulting (non-regular) cell complex is homeomorphic to $|\Sal(\A)|/W$ and 
is homotopy equivalent to $M(\A)/W$.

\begin{example}\label{construction}
Again, consider $D(\A_2)$.  If we label the two edges dual to $F_1$ and $F_2$ 
connected to $x(C_0)$ as $a$ and $b$, respectively, as in Figure 
\ref{orbitcomplex1}, then we see that there are two orbits of edges:  the 
$a$-type edges and the $b$-type edges.  This is consistent with the fact that 
the walls of the base chamber generate the reflection group.  If we identify all 
of the vertices and the edges based on type, then we will get a cell complex 
with one vertex and two edges.  See Figure \ref{orbitcomplex2}.  The resulting 
orbit complex will also have one $2$-cell bounded by the two edges.  In this 
special case, the resulting cell complex is homotopy equivalent to $F_3(\C)$.

\bigskip  

We now want to consider an orientation of the edges of $D(\A_2)$.  If $F$ is a 
codimension $1$ face of $\A_2$, then $F$ is a face of exactly two chambers, say 
$C$ and $C'$.  So, $(F,C)$ and $(F,C')$ are both elements of $\Sal(\A_2)$.  
Also, $|(F,C)|$ and $|(F,C')|$ are both edges of $|\Sal(\A_2)|$ with $|(C,C)|$ 
and $|(C',C')|$ as their vertices.  We orient each edge $|(F,C)|$ of 
$|\Sal(\A)|$ so that the edge points toward the vertex $|(C,C)|$.  It is easily 
verified that each $2$-cell of $|\Sal(\A_2)|$ is bounded by six edges.  We claim 
that the orientation of the edges of $D(\A_2)$ given in Figure 
\ref{orbitcomplex1} is the same as that of the $2$-cell $|(\{0\},C_0)|$.  This 
orientation determines the relation 
	$$aba = bab.$$
We could have chosen any of the $2$-cells of $|\Sal(\A_2)|$ to determine the 
orientation of the edges of $D(\A_2)$.  All of the $2$-cells would determine the 
same relation given by $|(\{0\},C_0)|$.  The relation $aba = bab$ is the 
relation in the presentation for the fundamental group of the orbit 
complex $|\Sal(\A_2)|/S_3$.  It is also the relation in the presentation for the 
braid group on three strands, which is isomorphic to the fundamental group of 
the unlabelled configuration space $F_3(\C)$.  Figure \ref{strands1} shows two 
braids, each on three strands.  The braid labelled $a$ corresponds to the edge 
$a$ of $D(\A_2)$.  Likewise, the braid labelled $b$ corresponds to the edge $b$ 
of $D(\A_2)$.  Then Figure \ref{strands2} captures the identity $aba = bab$.  
\begin{figure}[h]
\begin{center}

\includegraphics[width=6cm]{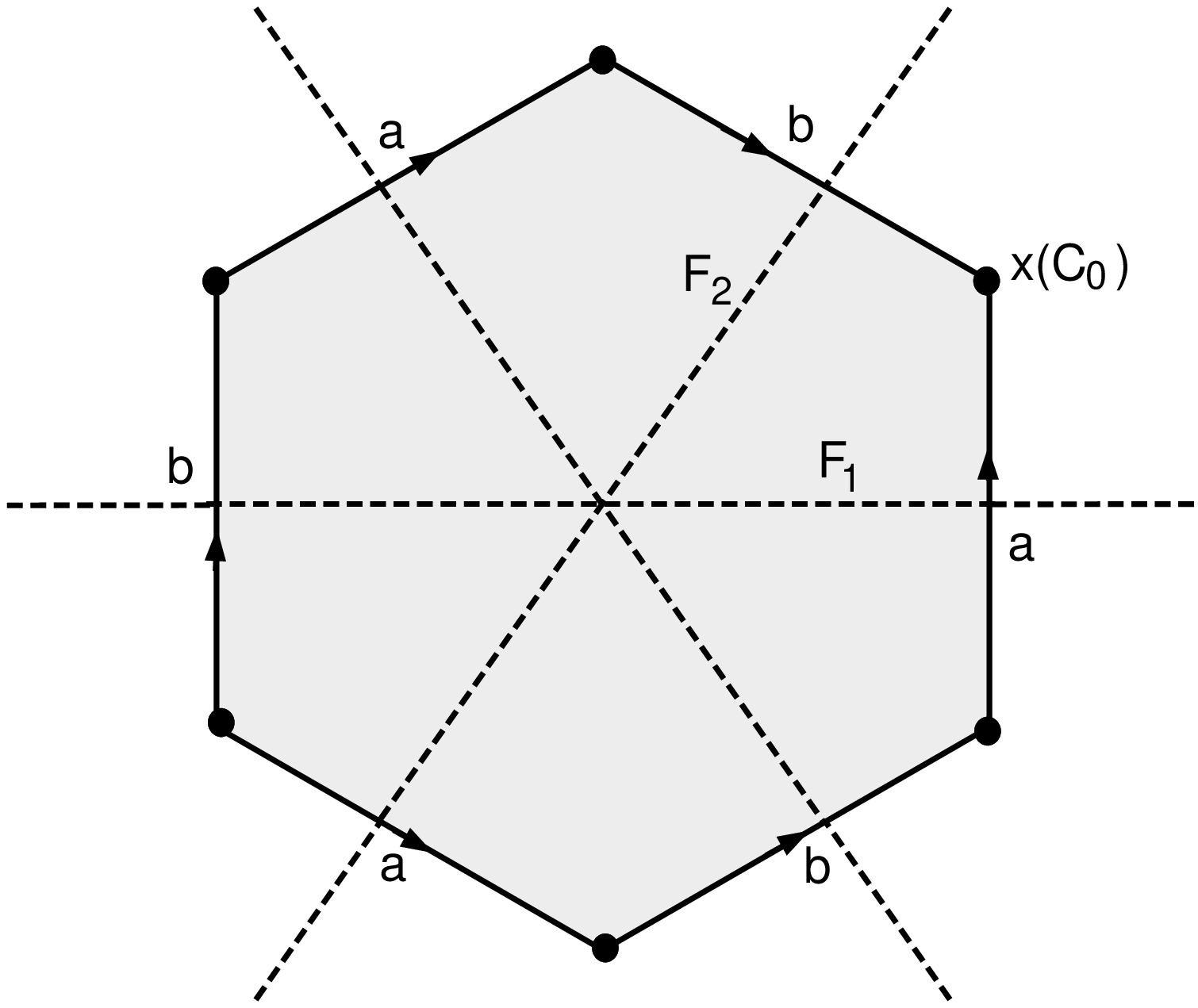}

\caption{Construction of $|\Sal(\A_2)|/S_3$ from $D(\A_2)$}
\label{orbitcomplex1}
\end{center}
\end{figure}

\begin{figure}[h]
\begin{center}

\includegraphics[width=6cm]{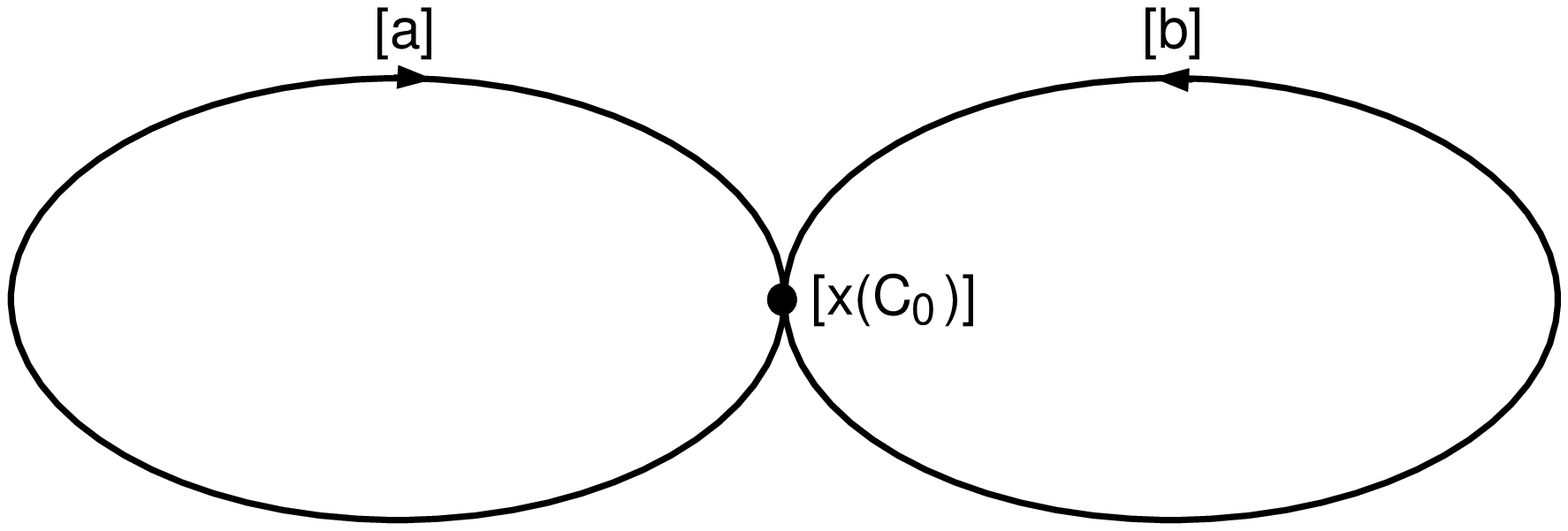}

\caption{The orbit complex $|\Sal(\A_2)|/S_3$}
\label{orbitcomplex2}
\end{center}
\end{figure}

\begin{figure}[h]
\begin{center}

\includegraphics[width=7.5cm]{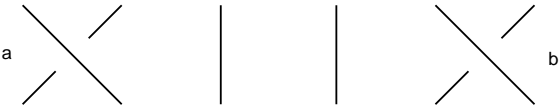}

\caption{Braids on three strands}
\label{strands1}
\end{center}
\end{figure}

\begin{figure}[h]
\begin{center}

\includegraphics[width=7.5cm]{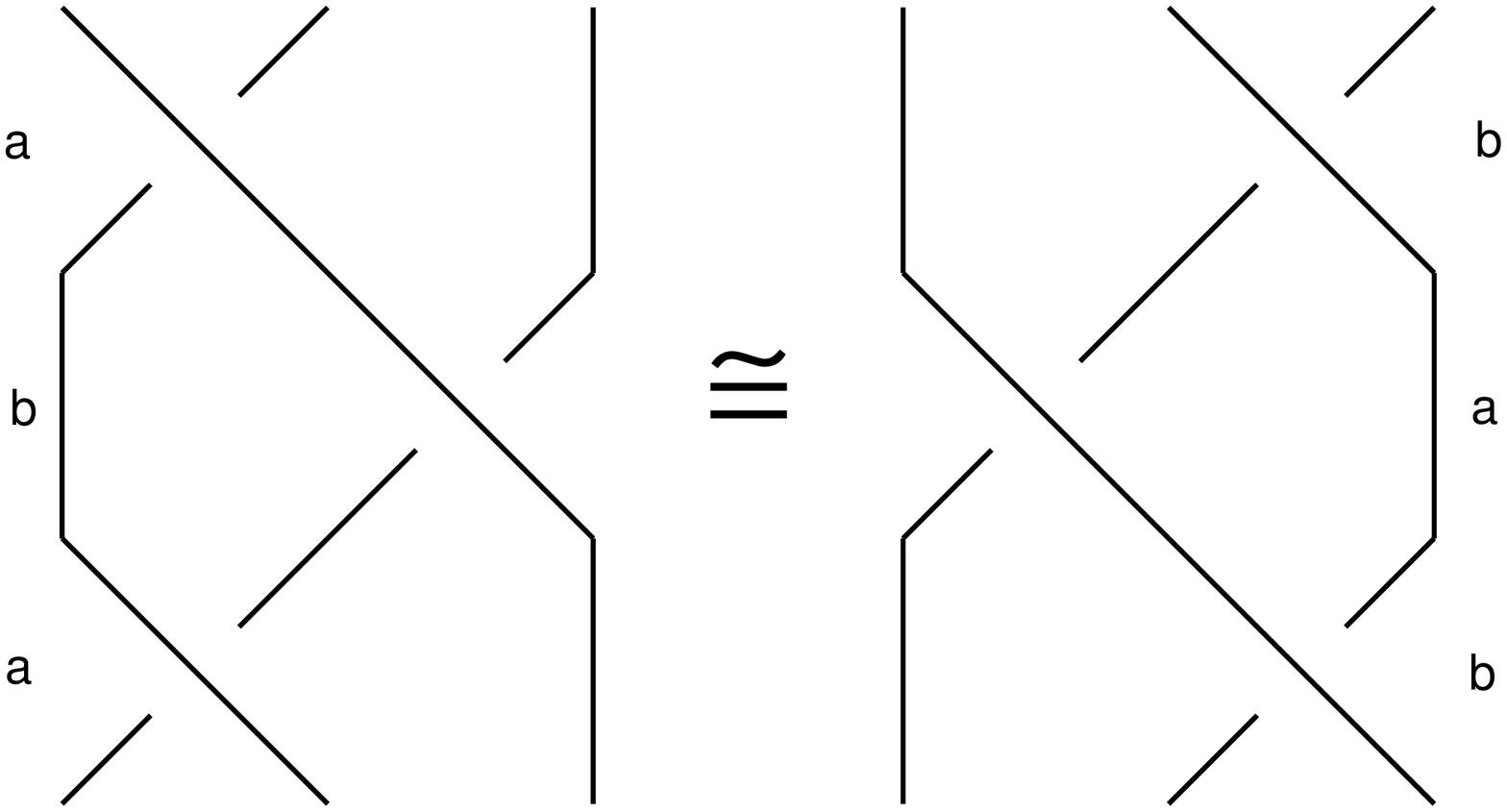}

\caption{The identity $aba=bab$}
\label{strands2}
\end{center}
\end{figure}
\end{example}
 
Now, let us provide an indication as to why this construction works.  First, 
note that if $C$ is a chamber of $\A$, then $|(C,C)|$ and $|(\{0\},C)|$ are both 
cells of $|\Sal(\A)|$.  Also, recall that cells of the form $|(C,C)|$ are the 
vertices of $|\Sal(\A)|$.  And cells of the form $|(\{0\},C)|$ are the top 
dimensional cells of $|\Sal(\A)|$.  In the case where $\A$ is an arrangement in 
$\R^2$, the top dimensional cells are $2$-cells.

\bigskip

The action of $W$ on $|\Sal(\A)|$ is given by
	$$w|(F,C)| = |(wF,wC)|\text{,}$$
where $w \in W$.  So, $W$ acts on the vertices of $|\Sal(\A)|$ in the same way 
that $W$ acts on the vertices of $|\Sal'(\A)|$.  That is, $w \in W$ acts on 
$|(C,C)| \in |\Sal(\A)|$ via
	$$w(|(C,C)|) = w(x(C))+iw(x(C)).$$
Since $W$ acts transitively on chambers, it follows that there is one orbit of 
cells of type $|(C,C)|$.  Also, since $w\{0\} = \{0\}$, there is one orbit of 
cells of type $|(\{0\},C)|$.  Thus, $|\Sal(\A)|/W$ has one vertex and one top 
dimensional cell.

\bigskip

For a $2$-dimensional arrangement $\A$, we can think of each vertex of $D(\A)$ 
as a vertex of $|\Sal(\A)|$, which is of the type $|(C,C)|$, and the $2$-cell of 
$D(\A)$ as the orbit of the top dimensional cells of $|\Sal(\A)|$, which are of 
the type $|(\{0\},C)|$.

\bigskip

Let $F$ be a $1$-codimensional face of an arrangement $\A$ in $\R^2$.  Then $F$ 
is a face of exactly two chambers, say $C$ and $C'$.  So, $(F,C)$ and 
$(F,C')$ are both elements of $\Sal(\A)$.  We claim that $|(F,C)|$ and 
$|(F,C')|$ are in the same orbit under the action of the reflection group.  
Indeed, reflection across the hyperplane containing $F$ will carry $|(F,C)|$ to 
$|(F,C')|$.  We can think of each edge in $D(\A)$ as a pair of edges of 
$|\Sal(\A)|$ of the type described above.  That is, the edge of $D(\A)$ that is 
dual to the face $F$ represents the identified edges $|(F,C)|$ and $|(F,C')|$ of 
$|\Sal(\A)|$.

\bigskip

For $2$-dimensional arrangements, we can think of $D(\A)$ as a ``partially 
identified" $|\Sal(\A)|$, where all of the top-dimensional cells have been 
identified and some of the edges have been identified.  The remaining 
identifications of edges of $|\Sal(\A)|$ can be shown to arise from 
identifications of the edges of $D(\A)$ arising from the action of $W$ as 
described above.

\bigskip

Let $\A$ be a reflection arrangement with group $W$.  Fix a chamber $C_0$.  Let 
$H_1, \ldots, H_n$ be the walls of $C_0$.  Then $W$ is generated by the 
reflections $s_{H_1}, \ldots, s_{H_n}$, according to Lemma \ref{L1}.  Also, 
according to \cite{HU}, $W$ has presentation
$$\langle s_{H_1}, \ldots, s_{H_n}: (s_{H_i})^2=1, s_{H_i}s_{H_j}s_{H_i} \cdots s_{H_i} 
= s_{H_j}s_{H_i}s_{H_j} \cdots s_{H_j} \rangle$$
where $m_{ij}$ factors appear on each side of the second relation for some 
integers $m_{ij}$ with $i \neq j$.  The $m_{ij}$ are easily determined by the 
``Coxeter graph" associated with the type of $W$.

\bigskip

The following well-known result given in \cite{BRI} is easy to see from 
the description of $|\Sal(\A)|/W$ in terms of the dual complex.  Example 
\ref{construction} is the special case with $W$ of type $A_2$. 

\begin{theorem}
Let $\A$ be a reflection arrangement with group $W$.  Then $\pi_1(M(\A)/W)$ has 
presentation
$$\langle g_1, \ldots, g_n: g_ig_jg_i \cdots g_i = g_jg_ig_j \cdots g_j \rangle$$
with $m_{ij}$ factors on each side of the relation. \hfill $\Box$
\end{theorem}
  
\end{chapter}


\end{document}